\documentclass[11pt,           % in 11pt fonts
final,                         % start in draft mode
english                       % default language
]{article}

\usepackage[english,strings]{babel}     % with explicit language
\usepackage{amsmath}           % ams mathematical stuff
\usepackage{amssymb}           % ams symbols
\usepackage[utf8]{inputenc}    % smart input of funny chars
\usepackage[T1]{fontenc}       % also for the font encoding
\usepackage{exscale}           % large summation signs in 11pt
\usepackage[sort]{cite}        % nicer citations
\usepackage{fancyhdr}          % nicer headers
\usepackage[a4paper]{geometry} % geometry of page layout
\usepackage{xspace}            % better spacing after macros
\usepackage{bbm}               % nicer bbm fonts
\usepackage{mathtools}         % some nice tricks for stackrel
\usepackage[amsmath,thmmarks,hyperref]{ntheorem} % nicer theorems
\usepackage{mathrsfs}          % nice calligraphic font
\usepackage{paralist}          % better enumerate lists
\usepackage{stmaryrd}
\usepackage{tikz}              % for commutative diagrams and stuff
\usepackage[expansion=false    % no font expansion
           ]{microtype}        % only protrusion
\usepackage[nottoc]{tocbibind} % refs and index in the toc
\usepackage[backref=page,      % backrefs in the bibliography
           final=true,         % always treat as final
           plainpages=false,              % whatever that means
           pdfpagelabels=true,            % strange options
           pdfencoding=auto,              % getting worse
           unicode=true,                  % unicode is always nice
           hypertexnames=true,            % needs both to get index *and*
           naturalnames=true
           ]{hyperref}         % hyperrefs are cool!
\usepackage{tikz-cd}           % f??r kommutative Diagramme
\usepackage[shortcuts]{extdash}  % nonbreaking hyphens

\numberwithin{equation}{section}

\allowdisplaybreaks

\let\originalleft\left
\let\originalright\right
\renewcommand{\left}{\mathopen{}\mathclose\bgroup\originalleft}
\renewcommand{\right}{\aftergroup\egroup\originalright}

\newcommand{\lemmachairxname}{Lemma}
\newcommand{\propositionchairxname}{Proposition}
\newcommand{\theoremchairxname}{Theorem}
\newcommand{\corollarychairxname}{Corollary}
\newcommand{\definitionchairxname}{Definition}
\newcommand{\claimchairxname}{Claim}
\newcommand{\examplechairxname}{Example}
\newcommand{\remarkchairxname}{Remark}
\newcommand{\questionchairxname}{Question}
\newcommand{\conjecturechairxname}{Conjecture}
\newcommand{\exercisechairxname}{Exercise}
\newcommand{\maintheoremchairxname}{Main Theorem}
\newcommand{\notationchairxname}{Notation}
\newcommand{\proofchairxname}{Proof}
\newcommand{\subproofchairxname}{Proof}
\newcommand{\hintchairxname}{Hint}

\StartBabelCommands{english}{extras}
\SetString{\lemmachairxname}{Lemma}
\SetString{\propositionchairxname}{Proposition}
\SetString{\theoremchairxname}{Theorem}
\SetString{\corollarychairxname}{Corollary}
\SetString{\definitionchairxname}{Definition}
\SetString{\claimchairxname}{Claim}
\SetString{\examplechairxname}{Example}
\SetString{\remarkchairxname}{Remark}
\SetString{\questionchairxname}{Question}
\SetString{\conjecturechairxname}{Conjecture}
\SetString{\exercisechairxname}{Exercise}
\SetString{\maintheoremchairxname}{Main Theorem}
\SetString{\notationchairxname}{Notation}
\SetString{\proofchairxname}{Proof}
\SetString{\subproofchairxname}{Proof}
\SetString{\hintchairxname}{Hint}
\EndBabelCommands

\theoremheaderfont{\normalfont\bfseries}
\theorembodyfont{\itshape}
\newtheorem{lemma}{\lemmachairxname}[section]

\newtheorem{proposition}[lemma]{\propositionchairxname}
\newtheorem{theorem}[lemma]{\theoremchairxname}
\newtheorem{corollary}[lemma]{\corollarychairxname}
\newtheorem{definition}[lemma]{\definitionchairxname}

\newtheorem{example}[lemma]{\examplechairxname}

\theorembodyfont{\itshape}
\theoremnumbering{Roman}

 \makeatletter
 \def\theorem@checkbold{}
 \makeatother
 
\theoremheaderfont{\scshape}
\theorembodyfont{\normalfont}
\theoremstyle{nonumberplain}
\theoremseparator{:}
\theoremsymbol{\hbox{$\boxempty$}}
\newtheorem{proof}{\proofchairxname}
\theoremsymbol{\hbox{$\triangledown$}}

\newcommand{\myname}{\textbf{Stefan Waldmann}}
\newcommand{\myemail}{\texttt{stefan.waldmann@mathematik.uni-wuerzburg.de}}

\newcommand{\AuthorOne}{\textbf{Daniela Kraus}}
\newcommand{\AuthorTwo}{\textbf{Oliver Roth}}
\newcommand{\AuthorThree}{\textbf{Matthias Schötz}}
\newcommand{\AuthorFour}{\myname}

\author{\AuthorOne\thanks{\AuthorEmailOne},
        \addtocounter{footnote}{2}
        \AuthorTwo\thanks{\AuthorEmailTwo},
        \AuthorThree\thanks{\AuthorEmailThree},
        { }\textbf{and}
        \AuthorFour\thanks{\AuthorEmailFour}
        \\[0.5cm]
        Julius Maximilian University of Würzburg \\
        Department of Mathematics \\
        97074 Würzburg \\
        Germany
}

\newcommand{\AuthorEmailOne}{\texttt{dakraus@mathematik.uni-wuerzburg.de}}
\newcommand{\AuthorEmailTwo}{\texttt{roth@mathematik.uni-wuerzburg.de}}
\newcommand{\AuthorEmailThree}{\texttt{matthias.schoetz@mathematik.uni-wuerzburg.de}}
\newcommand{\AuthorEmailFour}{\myemail}

\title{A Convergent Star Product on the Poincar\'e Disc}

\date{March 2018}

\newcommand{\ZZ}{\mathbbm{Z}}
\newcommand{\RR}{\mathbbm{R}}
\newcommand{\CC}{\mathbbm{C}}
\newcommand{\PP}{\mathbbm{P}}
\newcommand{\NN}{\mathbbm{N}}
\newcommand{\DD}{\mathbbm{D}}
\newcommand{\Tangent}{\textup{T}}

\newcommand{\Cotangent}{\textup{T}^*}
\newcommand{\MMup}{\,\mathcal{J}}
\newcommand{\MMdown}{\,\mathcal{J}_{\Discstd}}
\newcommand{\metric}{g}
\newcommand{\metricDouble}{\hat{\metric}}
\newcommand{\Smooth}{\Cinfty}
\newcommand{\algebra}[1]      {\mathscr{#1}}
\newcommand{\group}[1]        {\mathrm{#1}}
\newcommand{\acts}            {\mathbin{\triangleright}}
\newcommand{\racts}           {\mathbin{\triangleleft}}
\newcommand{\cc}[1]          {\overline{{#1}}}
\DeclarePairedDelimiter{\abs}{\lvert}{\rvert}
\DeclarePairedDelimiter{\norm}{\lVert}{\rVert}
\newcommand{\Unit}           {\mathbbm{1}}
\newcommand{\I}              {\mathrm{i}}
\newcommand{\E}              {\mathrm{e}}
\newcommand{\D}              {\mathop{}\!\mathrm{d}}
\newcommand{\argument}       {\,\cdot\,}
\newcommand{\tensor}[1][{}]           {\mathbin{\otimes_{\scriptscriptstyle{#1}}}}
\newcommand{\RE}             {\mathsf{Re}}
\newcommand{\IM}             {\mathsf{Im}}
\newcommand{\lie}[1]          {\mathfrak{#1}}
\newcommand{\Lie}   {\mathscr{L}}
\newcommand{\at}[1]          {\big|_{#1}}
\newcommand{\Stetig}         {\mathscr{C}}
\DeclareMathOperator{\Spec}     {\mathrm{Spec}}
\newcommand{\At}[1]          {\Big|_{#1}}
\DeclareMathOperator{\Unitary}  {\mathfrak{U}}
\newenvironment{propositionlist}{\begin{compactenum}[\itshape i.)]}{\end{compactenum}}
\DeclarePairedDelimiter{\ordinaryPOI}{\{}{\}}
\newcommand{\poi}[3][]{\ordinaryPOI[#1]{\,#2 \,,\, #3\,}}
\DeclarePairedDelimiter{\ordinaryKOM}{[}{]}
\newcommand{\kom}[3][]{\ordinaryKOM[#1]{\,#2 \,,\, #3\,}}
\DeclarePairedDelimiter{\ordinaryIP}{\langle}{\rangle}

\newcommand{\skal}[3][]{\ordinaryIP[#1]{#2 \,#1|\, #3}}
\DeclarePairedDelimiter{\ordinarySet}{\{}{\}}
\newcommand{\set}[3][\big]{\ordinarySet[#1]{\,#2 \;#1|\; #3\,}}
\newcommand{\lieaction}[1]{\xi_{#1}}
\newcommand{\Levelsetstd}{Z}
\newcommand{\LevelsetDouble}{\hat{Z}}
\newcommand{\Discstd}[1][n]{D_#1}
\newcommand{\Discext}[1][n]{D_{#1,\text{ext}}}
\newcommand{\DiscDouble}[1][n]{\hat{D}_#1}
\newcommand{\DiscDoubleDefStd}{\hat{D}^\std_n}
\newcommand{\DiscDoubleDefP}{\hat{D}^P_n}
\newcommand{\DiscDoubleDefQ}{\hat{D}^Q_n}
\newcommand{\chartstd}{\phi^\std}
\newcommand{\chartstdDouble}{\hat{\phi}^\std}
\newcommand{\chartPDouble}{\hat{\phi}^P}
\newcommand{\chartQDouble}{\hat{\phi}^Q}

\newcommand{\chartstdDoubleInv}{(\hat{\phi}^\std)^{-1}}
\newcommand{\chartPDoubleInv}{(\hat{\phi}^P)^{-1}}
\newcommand{\chartQDoubleInv}{(\hat{\phi}^Q)^{-1}}

\newcommand{\hOben}{h}
\newcommand{\gOben}{g}
\newcommand{\omegaOben}{\omega}
\newcommand{\piOben}{\pi}
\newcommand{\hUnten}{h_d}
\newcommand{\gUnten}{g_d}
\newcommand{\omegaUnten}{\omega_d}
\newcommand{\piUnten}{\pi_d}
\newcommand{\inclusionLevelsetstd}{\iota}
\newcommand{\inclusionLevelsetDouble}{\hat{\iota}}
\newcommand{\prodiscstd}{\textup{pr}}
\newcommand{\prodiscDouble}{\hat{\textup{pr}}}

\newcommand{\DiagOben}{\Delta}
\newcommand{\DiagObenLevelset}{\Delta_{\Levelsetstd}}
\newcommand{\DiagUnten}{\Delta_D}
\newcommand{\DiagUntenExt}{\Delta_{\textup{x}}}
\newcommand{\DiagCP}{\Delta_\PP}
\newcommand{\involutionAll}{\tau}
\DeclareMathOperator{\starTilde}{\tilde{\star}}
\newcommand{\starWick}[1][]{\starTilde_{#1}}
\newcommand{\starRed}[1][]{\star_{#1}}
\newcommand{\formalPS}[1]{[\![#1]\!]}

\newcommand{\monomial}[2]{\textup{d}_{#1,#2}}

\newcommand{\monomialDown}[2]{\textup{f}_{#1,#2}}
\newcommand{\monomialDownHat}[2]{\hat{\textup{f}}_{#1,#2}}
\newcommand{\monomialDownRed}[2]{\textup{f}_{\textup{r},#1,#2}}
\newcommand{\monomialDownRedHat}[2]{\hat{\textup{f}}_{\textup{r},#1,#2}}
\newcommand{\CoeffUnten}[2]{\textup{f}^{\,\prime}_{\textup{r},#1,#2}}

\newcommand{\Transpose}{T}
\newcommand{\embedInCPn}{\iota_{\PP}}
\newcommand{\embedInCPnExt}{\iota_{\textup{ext},\PP}}
\newcommand{\embedInCPnCpn}{\iota_{\PP\times\PP}}
\newcommand{\embedInExt}{\iota_{\textup{ext}}}
\newcommand{\Holomorphic}{\mathcal{O}}
\newcommand{\AnalyticOben}{\algebra{A}(\CC^{1+n})}
\newcommand{\PolynomeOben}[1][]{\algebra{P}(\CC^{1+n})^{#1}}

\newcommand{\PolynomeUnten}[1][]{\algebra{P}(\Discstd)^{#1}}
\newcommand{\AnalyticUnten}[1][n]{\algebra{A}(\Discstd[#1])}

\newcommand{\Cinfty}{\mathscr{C}^{\infty}}

\DeclareFontFamily{U}{FdSymbolF}{}
\DeclareFontShape{U}{FdSymbolF}{m}{n}{
    <-7.1> FdSymbolF-Book
    <7.1-> FdSymbolF-Book
}{}
\DeclareSymbolFont{delimiters}{U}{FdSymbolF}{m}{n}
\DeclareMathDelimiter{\llangle}{\mathopen}{delimiters}{"92}{delimiters}{"92}
\DeclareMathDelimiter{\rrangle}{\mathclose}{delimiters}{"98}{delimiters}{"98}

\newcommand{\seminormOben}[3][]{\norm[#1]{#3}_{\CC^{1+n},#2}}

\newcommand{\seminormUnten}[3][]{\norm[#1]{#3}_{\Discstd,#2}}
\newcommand{\std}{\textup{std}}
\newcommand{\HbarDef}{H}
\newcommand{\PreHilb}{\mathcal{H}}
\newcommand{\Adbar}{\mathcal{L}^*}
\newcommand{\rep}{\pi}
\newcommand{\GroupRep}{\textbf{U}}
\newcommand{\GroupRepDeriv}{\D\GroupRep}
\newcommand{\OpClosure}{{\textup{cl}}}
\newcommand{\invSpecial}{\Sigma}
\newcommand{\ReductionMap}[1]{\Psi_{#1}}
\newcommand{\ReductionMapInv}[1]{\Phi_{#1}}
\newcommand{\UnitHom}{\textup{Spec}}
\newcommand{\Characters}{\textup{Spec}^*}

\begin{document}

\selectlanguage{english}

%
% title page
%

\maketitle

%
% abstract
%

\begin{abstract}
    On the Poincar\'e disc and its higher-dimensional analogs one has a
    canonical formal star product of Wick type. We define a locally
    convex topology on a certain class of real-analytic functions on
    the disc for which the star product is continuous and converges as
    a series. The resulting Fr\'echet algebra is characterized
    explicitly in terms of the set of all holomorphic functions on an extended and doubled disc
    of twice the dimension endowed with the natural topology
    of locally uniform convergence. We discuss the holomorphic dependence on
    the deformation parameter and the positive functionals and their
    GNS representations of the resulting Fr\'echet algebra.
\end{abstract}

%
% table of contents
%

\tableofcontents
\newpage

%
% Introduction
%

\section{Introduction}
\label{sec:Introduction}

Deformation quantization comes in two principle flavours: formal
deformation quantization as introduced in \cite{bayen.et.al:1978a}
considers formal associative deformations of the algebra $\Cinfty(M)$
of smooth functions on a Poisson manifold $M$
in the sense of Gerstenhaber
\cite{gerstenhaber:1964a}. Here the existence as well as the
classification of such formal star products is well-understood and
established, see \cite{kontsevich:2003a} for the final case of Poisson
manifolds and e.g. \cite{fedosov:1994a, dewilde.lecomte:1983b,
  nest.tsygan:1995a, bertelson.cahen.gutt:1997a} for the symplectic
case. While the formal deformations are well-understood, they still
suffer from being not yet physically relevant: the deformation
parameter plays the role of Planck's constant $\hbar$ and thus should
be treated not as formal. This leads to the second flavour of
deformation quantization where one tries to find a more analytic
framework in such a way that the deformation of the algebra of
functions on the classical phase space depends continuously or even
analytically on $\hbar$. Here the situation is less clear as there is
not yet a general framework which would allow to prove general
existence or classification results. Instead, one has several
competing definitions of what a non-formal version of a formal star
product should be, most notably a $C^*$-algebraic approach as taken
e.g. in \cite{bieliavsky.gayral:2015a, natsume:2000a, rieffel:1993a,
  natsume.nest:1999a, natsume.nest.peter:2003a} where several classes
of examples are studied. There the notion of deformation is based on
continuous fields of $C^*$-algebras which satisfy the correct
asymptotic for $\hbar \longrightarrow 0$ on a sufficiently large
domain. While the $C^*$-algebraic environment is of course best suited
to purposes in quantum theory, the dependence on $\hbar$ is typically
not much better than continuous. Moreover, the construction of the
continuous fields is fairly complicated. As a final remark on these
constructions one should note that on the technical level, all of them
are based on certain integration procedures which are limited to
finite-dimensional phase spaces. Thus the integral formulas for the
non-formal products can not be generalized to field-theoretic
situations.

The alternative approach consists now in finding suitable subalgebras
of the formal star product algebra where the series defining the
product actually converges. To make sense out of such an attempt one
needs to specify the notion of convergence. Here a wide variety of
possibilities exists, ranging from pointwise or even weaker notions of
convergence to more uniform versions taking into account also
convergence of derivatives etc. In previous works, such subalgebras
with natural locally convex topologies have been identified and
studied in several examples, allowing also for infinite-dimensional
ones: based on the earlier works \cite{beiser.roemer.waldmann:2007a,
  omori.maeda.miyazaki.yoshioka:2000a} for finite dimensions, the case
of a locally convex Poisson vector space with continuous constant
Poisson structure was treated in detail in \cite{waldmann:2014a} and
specified further in \cite{schoetz.waldmann:2018a} for the case where
the underlying topology is pro-Hilbert. Moreover, in
\cite{esposito.stapor.waldmann:2017a} the case of linear Poisson
structures, i.e. the Poisson structure on the dual of a Lie algebra,
with the Gutt star product \cite{gutt:1983a} was studied in
detail. Also in this case infinite-dimensional Lie algebras can be
included once the Lie bracket has a certain continuity property always
satisfied in the finite-dimensional case. Finally, the first truly
curved example, the Poincar\'e disc $D_n$,
was studied in
\cite{beiser.waldmann:2014a, beiser:2011a} where the convergence of
the Wick star product obtained by reduction as in
\cite{bordemann.brischle.emmrich.waldmann:1996a,
  bordemann.brischle.emmrich.waldmann:1996b} is shown for a locally
convex topology based on the growth of the coefficients with respect
to a particular Schauder basis of functions on the disc.

It is this example which we investigate more closely: first, we
improve the results of \cite{beiser.waldmann:2014a} in so far as we
prove the continuity of the Wick product with respect to a slightly
coarser topology on the same underlying linear span of the (to become
a) Schauder basis. This gives a slightly larger completion to a
Fr\'echet algebra as before. The impact of this 
construction is now that we are able to explicitly characterize
this Fr\'echet algebra in terms of the F\'rechet algebra $\Holomorphic(\DiscDouble)$
of all holomorphic functions on a  complex manifold $\DiscDouble$, which is
obtained by doubling the Poincar\'e disc $\Discstd$ in a geometrically
faithful way.
The original functions on the disc $\Discstd$ are then obtained by restricting  these
holomorphic functions on $\DiscDouble$ to the ``diagonal'' disc $D_n$.

The explicit characterization of the completion opens the doors to
many further investigations of this example. In particular, we are
able to determine the (classically) positive functionals and show that
the evaluation functionals on the disc stay positive with respect to
the Wick star product, too. In the corresponding GNS representation we
are able to prove
(essential) self-adjointness of many elements of the
algebra, among them the generators of the $\mathfrak{su}(1, n)$
symmetry.

As an open question in the one--dimensional case $n=1$ we mention that the convergent star product we
construct will immediately pass to Riemann surfaces of higher
genus \emph{provided} we have enough invariant functions under the
corresponding Fuchsian group on the unit disk $D_1\cong\DD:=\{z\in
  \CC : |z|<1\}$ in our algebra. At the moment, this is
not completely clear and therefore we postpone this question to a
future project. In any case, it is fairly easy to see that we have
many bounded functions in the completion even though in the linear
span of the Schauder basis there are only the constant functions
bounded. In any case, this question becomes an entirely classical
question independent of the deformed product: one needs to understand
the holomorphic functions on the extended doubled disc. This will also
allow a more detailed comparison with the very recent approach of
Bieliavsky \cite{bieliavsky:2017a} to quantum Riemann
surfaces where the $C^*$-algebraic approach is used.

Ultimately, we believe that this example contains many options to
approach also more complicated situations of Wick type quantization of
certain K\"ahler manifolds. The rich geometry of the doubled and
extended disc will be present in further sufficiently symmetric cases
as well.

The paper is organized as follows: in Section~\ref{sec:Construction}
we recall the construction of the Wick star product on the Poincar\'e
disc by phase space reduction and describe the geometry of the
extended disc. In Section~\ref{sec:ConstructionFrechetAlgebra} we
define the topology based on the coefficient expansion with respect to
the basis functions on the disc. The main result is then the
characterization of the completion and the convergence of the Wick
star product as a series for all elements in the completion.  In the
last section we investigate many properties of the Frech\'et algebra we
obtained: the dependence on the deformation parameter $\hbar$ turns
out to be holomorphic on
$\mathbbm{C} \setminus \{0, - \frac{1}{2m } \; | \; m \in \mathbbm{N}\}$
and we have an asymptotic in the Frech\'et topology of the product for
$\hbar \longrightarrow 0^+$ establishing the correct semi-classical
limit. Moreover, we determine the positive functionals and their GNS
representations showing the essential self-adjointness of many
important algebra elements in all continuous representations. Finally, in the
case of the disc, i.e. in complex dimension $n = 1$, we have an
additional discrete symmetry not present in higher dimensions.

%
% Acknowledgement
%

\vspace{1em}
\noindent
\textbf{Acknowledgement:}

\noindent
Two of us (M. Schötz, S. Waldmann) would like to thank Simone Gutt, UlB, for valueable discussions.

%
% Preliminaries: The Construction of the Star Product
%

\section{Preliminaries: The Construction of the Star Product}
\label{sec:Construction}

In this section we are going to present the essential steps of the
construction of the star product on the Poincar\'e disc, introduce our
notation, and also discuss some additional structures that will be
helpful later on. Essentially everything is either standard concerning
the Marsden-Weinstein reduction or can be found in the previous works
\cite{bordemann.brischle.emmrich.waldmann:1996a,
  bordemann.brischle.emmrich.waldmann:1996b, beiser.waldmann:2014a,
  beiser:2011a} concerning the construction of the star product, see
also \cite{cahen.gutt.rawnsley:1995a, cahen.gutt.rawnsley:1994a,
  cahen.gutt.rawnsley:1993a, cahen.gutt.rawnsley:1990a} for yet
another approach to this star product on the Poincar\'e disc.

First we are going to introduce various manifolds and maps between
them: In addition to the original manifold $\CC^{1+n}$, the level set
$\Levelsetstd$, and the reduced manifold $\Discstd$ from the
Marsden-Weinstein reduction this includes extensions of these to
complex manifolds $\CC^{1+n}\times\CC^{1+n}$, $\LevelsetDouble$, and
$\DiscDouble$ that will allow us to define certain spaces of
real-analytic functions on $\CC^{1+n}$ and $\Discstd$.

Next we construct the classical Poisson $^*$\=/algebra on the space
$\Smooth(\Discstd)$ of all complex-valued smooth functions on
$\Discstd$, its subalgebras
$\AnalyticUnten$ and $\PolynomeUnten$ of analytic and polynomial functions, as well as the classical
reduction map $\ReductionMap{0}$ from functions on $\CC^{1+n}$ to such
on $\Discstd$.

The deformed quantum algebra on $\Discstd$ can then be obtained by a
similar reduction procedure starting from the space of polynomials on
$\CC^{1+n}$ with a Wick-type star product.

%
% The Poincar\'e disc $\Discstd$
%

\subsection{\texorpdfstring{The Poincar\'e disc $\Discstd$}{The Poincar\'e disc Dn}}
\label{subsection:geometry}
The whole geometric construction can be summarized by the following
commutative diagram (the detailed description follows).  The upper
horizontal row consists of complex manifolds, each of which is
equipped with an anti-holomorphic involution $\involutionAll$ and a
smooth action of the Lie group
$\group{U}(1, n) = \group{U}(1) \times \group{SU}(1, n)$ by
holomorphic automorphisms that commute with $\involutionAll$. The
arrows between them are holomorphic maps and equivariant with respect
to the $\group{U}(1,n)$-action and the involution $\involutionAll$.
All other objects are (at least) smooth manifolds equipped with a
smooth action of $\group{U}(1, n)$ and all other arrows are
$\group{U}(1, n)$-equivariant smooth maps:
\begin{center}
  \begin{tikzcd}
    \Big.\CC^{1+n}\times\CC^{1+n} \quad
    &&
    \arrow[swap]{ll}{ \hat{\iota} }
    \Big.\quad \LevelsetDouble \quad
    \arrow{rr}{\prodiscDouble }
    &&
    \Big.\quad \DiscDouble \quad
    \arrow{rr}{ \embedInCPnCpn }
    &&
    \Big.\quad \CC\PP^n\times\CC\PP^n
    \\
    &&
    &&
    &
     \arrow[swap]{lu}{\DiagUntenExt}
    \Big.\quad \Discext \quad
     \arrow{rd}{ \iota_{\textup{x},\PP} }
    &
    \\
    \Big.\quad \CC^{1+n}\quad
    \arrow[swap]{uu}{\DiagOben}
    &&
    \arrow[swap]{ll}{ \inclusionLevelsetstd }
    \Big.\quad \Levelsetstd \quad
    \arrow[swap]{uu}{\DiagObenLevelset}
    \arrow{rr}{ \prodiscstd }
    &&
    \Big.\quad \Discstd \quad
    \arrow[swap]{uu}{\DiagUnten}
    \arrow{ru}{ \embedInExt }
    \arrow{rr}{ \embedInCPn }
    &&
    \Big.\quad \CC\PP^n \quad
    \arrow[swap]{uu}{\DiagCP}
  \end{tikzcd}
\end{center}

%
% Bottom row
%

\subsubsection*{Bottom row:}
The Poincar\'e disc $\Discstd$ can be constructed as a quotient of a
subset $\Levelsetstd$ of $\CC^{1+n}$. This procedure will be linked to
Marsden-Weinstein reduction in Section~\ref{subsection:classically}.

On the smooth manifold $\CC^{1+n}$ with standard (complex) coordinates
$z^0,\dots,z^n\colon \CC^{1+n} \to \CC$, the Lie group
$\group{U}(1,n)$ acts from the left via $U\acts r := Ur$ for
$U\in \group{U}(1,n)$ and $r\in \CC^{1+n}$. Define
\begin{equation}
    \label{eq:KaehlerPotential}
    \metric := h_{\mu\nu} z^\mu \cc{z}^\nu \in\Smooth(\CC^{1+n}),
\end{equation}
where $h_{00} := -1$, $h_{ii} := 1$ for $i \in \{1, \ldots, n\}$ and
$h_{\mu\nu} := 0$ otherwise, then
\begin{equation}
    \label{eq:TheLevelset}
    \Levelsetstd
    :=
    \metric^{-1}\big(\{-1\}\big)
    =
    \set[\Big]{
      r\in \CC^{1+n}
    }{
      \abs{z^0(r)}^2 = 1 + \sum\nolimits_{i=1}^n \abs{z^i(r)}^2
    }
\end{equation}
is the orbit of $(1, 0, \ldots, 0)^\Transpose \in \CC^{1+n}$ under the
$\group{U}(1, n)$-action and a submanifold of $\CC^{1+n}$.
Consequently, the action of $\group{U}(1, n)$ can be restricted to
$\Levelsetstd$ and the canonical inclusion $\inclusionLevelsetstd$ of
the subset $\Levelsetstd$ in $\CC^{1+n}$ is of course
$\group{U}(1, n)$-equivariant.

Next consider the (compact) Lie subgroup
$\group{U}(1) \cong \set{\E^{\I \phi} \Unit_{1+n}}{\phi \in \RR}
\subseteq \group{U}(1, n)$ and construct the quotient manifold
\begin{equation}
    \label{eq:TheDisc}
    \Discstd
    :=
    \Levelsetstd \big/ \group{U}(1).
\end{equation}
As the $\group{U}(1)$-subgroup lies in the center of
$\group{U}(1, n)$, the action of $\group{U}(1, n)$ remains
well-defined on $\Discstd$ (of course, only the complementary
$\group{SU}(1, n)$-subgroup acts non-trivially) and is still
transitive. The canonical projection
$\prodiscstd\colon \Levelsetstd\to \Discstd$ then is an
$\group{U}(1, n)$-equivariant smooth map.

Finally, $\Discstd$ can be embedded injectively in $\CC\PP^n$ via
$\embedInCPn\colon \Discstd\to\CC\PP^n$,
\begin{equation}
    \label{eq:InjectDiscToCPn}
    [r]_{\group{U}(1)}
    \mapsto
    \embedInCPn\big([r]_{\group{U}(1)}\big)
    :=
    [r]_{\CC_*},
\end{equation}
where we denote the equivalence classes with respect to the different
group actions by $[\argument]_{\group{U}(1)}$ and
$[\argument]_{\CC_*}$, respectively, to indicate the different
equivalence relations.  This embedding is also
$\group{U}(1, n)$-equivariant with respect to the
$\group{U}(1, n)$-action on $\CC\PP^n$ inherited from
$\CC^{1+n}\setminus\{0\}$, i.e. $U\acts[r] := [Ur]$ for
$U\in\group{U}(1, n)$ and $[r] \in \CC\PP^n$. Note that $\CC\PP^n$ is
the disjoint union
\begin{equation}
    \label{eq:DecomposeCPn}
    \CC\PP^n
    =
    \set[\big]{[r]\in\CC\PP^n}{\metric(r)<0}
    \cup
    \set[\big]{[r]\in\CC\PP^n}{\metric(r)=0}
    \cup
    \set[\big]{[r]\in\CC\PP^n}{\metric(r)>0},
\end{equation}
and that the image of $\embedInCPn$ is
$\set{[r] \in \CC\PP^n}{\metric(r) < 0}$. From $\CC\PP^n$, the disc
$\Discstd$ inherits the structure of a complex $n$-dimensional
manifold, and can even be covered by a single holomorphic chart
$\chartstd = (w^1, \ldots, w^n)^\Transpose \colon \Discstd\to
\DD_n\subseteq \CC^n$, where $\DD_n$ is the  $n$-dimensional
  polydisc  $\DD \times \cdots \times \DD$
and
\begin{equation}
    \label{eq:StandardChart}
    w^i([r]) := \frac{z^i(r)}{z^0(r)}
\end{equation}
for all $i \in \{1, \ldots, n\}$ and $[r]\in \Discstd$. This chart
actually is a biholomorphic mapping from $\Discstd$ to $\DD_n$ and
the action of $\group{U}(1, n)$ on $\Discstd$ in this chart is
described by M\"obius transformations
\begin{equation}
    \label{eq:MoebiusTrafo}
    \chartstd\big(U\acts [r]\big)
    =
    \chartstd([Ur])
    =
    \frac{c+A\chartstd([r])}{\alpha+b\cdot\chartstd([r])}
    \quad
    \textrm{ for }
    \quad
    [r]\in \Discstd \text{ and }
    U
    =
    \begin{pmatrix}
        \alpha & b^\Transpose \\
        c & A
    \end{pmatrix}
    \in \group{SU}(1, n),
\end{equation}
with $\alpha\in \CC$, $b, c \in \CC^n$, and $A \in \CC^{n\times n}$
such that $U\in\group{SU}(1, n)$.

%
% Top row
%

\subsubsection*{Top row:}
As we are interested in a non-formal star product on $\Discstd$, which
cannot be constructed on all smooth functions but only on certain
analytic functions on $\Discstd$, we have to extend the above
construction in a certain way that allows us to describe these
analytic functions appropriately, i.e. as pullback with a smooth map
$\DiagUnten\colon \Discstd \to \DiscDouble$ of holomorphic functions
on some complex manifold $\DiscDouble$. This leads to the top row of
the diagram above.

On the complex manifold $\CC^{1+n} \times \CC^{1+n}$ with standard
holomorphic coordinate functions
$x^0, \ldots, x^n, \allowbreak y^0, \ldots, y^n$, the group
$\group{U}(1, n)$ acts from the left via
$U\acts (p, q) := (Up, \cc{U}q)$ for all $U \in \group{U}(1, n)$ and
$p, q \in \CC^{1+n}$. Define the $\group{U}(1, n)$-equivariant
anti-holomorphic involution $\involutionAll(p, q) := (\cc{q}, \cc{p})$
on $\CC^{1+n}\times\CC^{1+n}$ as well as the
$\group{U}(1, n)$-equivariant diagonal inclusion
$\DiagOben\colon \CC^{1+n}\to \CC^{1+n}\times\CC^{1+n}$
\begin{equation}
    \label{eq:DiagonalMap}
    r\mapsto \DiagOben(r) := (r, \cc{r}),
\end{equation}
then $\DiagOben$ describes a diffeomorphism from $\CC^{1+n}$ to
$\set{(p, q) \in \CC^{1+n} \times \CC^{1+n}}{\involutionAll(p, q) =
  (p, q)}$.  Moreover, define
\begin{equation}
    \label{eq:DefmetricDouble}
    \metricDouble
    :=
    h_{\mu\nu} x^\mu y^\nu
    \in
    \Holomorphic(\CC^{1+n}\times\CC^{1+n}),
\end{equation}
then $\metricDouble\circ\DiagOben = \metric$ holds. Let
\begin{equation}
    \label{eq:DoubleLevelSet}
    \LevelsetDouble
    :=
    \metricDouble^{-1}\big(\{-1\}\big),
\end{equation}
then $\LevelsetDouble$ is a holomorphic submanifold on
$\CC^{1+n} \times \CC^{1+n}$, the $\group{U}(1, n)$-action as well as
the anti-holomorphic involution $\involutionAll$ can be restricted to
$\LevelsetDouble$ and so the canonical inclusion
$\inclusionLevelsetDouble\colon \LevelsetDouble \to
\CC^{1+n}\times\CC^{1+n}$
is $\group{U}(1, n)$- and $\involutionAll$-equivariant.  Define
$\DiagObenLevelset\colon \Levelsetstd\to \LevelsetDouble$ as the
restriction of $\DiagOben$, then $\DiagObenLevelset$ is still
$\group{U}(1, n)$-equivariant, the left rectangle of the above diagram
commutes and $\DiagObenLevelset$ describes a diffeomorphism from
$\Levelsetstd$ to
$\set{(p, q) \in \LevelsetDouble}{\involutionAll(p, q) = (p, q)}$.

The action of the Lie subgroup
$\group{U}(1) \cong \set{\E^{\I \phi} \Unit_{1+n}}{\phi \in \RR}
\subseteq \group{U}(1, n)$
on $\LevelsetDouble$ can be extended to a holomorphic action of
$\CC_*$ via $z\acts (p, q) := (zp, q/z)$ for all
$(p, q)\in\Levelsetstd$ and $z \in \CC_*$. As this action is free and
proper, we can construct the holomorphic quotient manifold
\begin{equation}
    \label{eq:DiscDouble}
    \DiscDouble := \LevelsetDouble \big/ \CC_*.
\end{equation}
The $\CC_*$-action on $\LevelsetDouble$ commutes with the
$\group{U}(1, n)$-action, so the $\group{U}(1, n)$-action remains
well-defined on $\DiscDouble$ (again, only the
$\group{SU}(1, n)$-subgroup acts non-trivially) and the canonical
projection $\prodiscDouble\colon \LevelsetDouble\to \DiscDouble$ is
$\group{U}(1, n)$-equivariant. Moreover, there is a unique
anti-holomorphic involution $\involutionAll$ on $\DiscDouble$ such
that $\prodiscDouble$ becomes also $\involutionAll$-equivariant,
namely $\involutionAll\big([p, q]\big) := [(\cc{q}, \cc{p})]$, which
is indeed well-defined and commutes with the action of
$\group{U}(1, n)$.  The map $\DiagObenLevelset$ remains well-defined
on the quotients, so
$\DiagUnten\colon \Discstd \to \DiscDouble$,
\begin{equation}
    \label{eq:DiagonalUnten}
    [r]\mapsto\DiagUnten([r]) := \DiagObenLevelset(r)
\end{equation}
is well-defined and a smooth and $\group{U}(1, n)$-equivariant map. It
is now easy to see that the central rectangle in the above diagram
commutes as well.

Finally, we can again embed $\DiscDouble$ injectively in
$\CC\PP^n \times \CC\PP^n$ via
$\embedInCPnCpn\colon \DiscDouble \to \CC\PP^n \times \CC\PP^n$
\begin{equation}
    \label{eq:DoubleDiscToDoubleCPn}
    [p, q]_{\CC_*} \mapsto \embedInCPnCpn\big([p, q]_{\CC_*}\big)
    :=
    \big([p]_{\CC_*}, [q]_{\CC_*}),
\end{equation}
which is equivariant with respect to the $\group{U}(1, n)$-action on
$\CC\PP^n \times \CC\PP^n$ defined as
$U\acts\big([p], [q]) := \big([Up], [\cc{U}q]\big)$ for all
$[p], [q] \in \CC\PP^n$ and $U \in \group{U}(1, n)$. It is also
$\involutionAll$-equivariant if one defines the anti-holomorphic
involution $\involutionAll$ on $\CC\PP^n \times \CC\PP^n$ as
$\involutionAll([p], [q]) := ([\cc{q}], [\cc{p}])$.  The image of
$\embedInCPnCpn$ in $\CC\PP^n \times \CC\PP^n$ is then
$\set{\big([p], [q]\big) \in \CC\PP^n \times
  \CC\PP^n}{\metricDouble(p, q) \neq 0}$
and pulling back the usual charts from $\CC\PP^n \times \CC\PP^n$ to
$\DiscDouble$ yields suitable holomorphic charts on $\DiscDouble$. We
especially define the standard chart
$\chartstdDouble = (u^1, \ldots, u^n, v^1, \ldots, v^n)^\Transpose
\colon \DiscDoubleDefStd \to C^\std\subseteq \CC^n\times\CC^n$ by
\begin{equation}
    \label{eq:standardcoord}
    u^i([p, q])
    :=
    \frac{x^i(p, q)}{x^0(p, q)}
    \quad
    \textrm{and}
    \quad
    v^i([p, q])
    :=
    \frac{y^i(p, q)}{y^0(p, q)}
\end{equation}
for all $i \in \{1, \ldots, n\}$ and $[p, q] \in \DiscDouble$, with
domain
\begin{equation}
  \DiscDoubleDefStd := \set[\big]{[p, q] \in \DiscDouble}{x^0(p, q)
  \neq 0 \textrm{ and } y^0(p, q) \neq 0}
\end{equation}
and image
$C^\std := \set[\big]{(p, q) \in \CC^n \times \CC^n}{p \cdot q \neq
  1}$.
Moreover, let $\DiagCP\colon \CC\PP^n \to \CC\PP^n \times \CC\PP^n$ be
the $\group{U}(1, n)$-equivariant diagonal inclusion
\begin{equation}
    \label{eq:YetAnotherDiagonal}
    [r] \mapsto \DiagCP\big([r]\big) := \big([r], [\cc{r}]),
\end{equation}
then the right rectangle in the above diagram commutes.

%
% The extended disc $\Discext$
%

\subsubsection*{\texorpdfstring{The extended disc $\Discext$:}{The extended disc Dn,ext:}}
Comparing the images of the embeddings of $\Discstd$ and $\DiscDouble$
in $\CC\PP^n$ and $\CC\PP^n\times\CC\PP^n$, respectively, shows that
the image of $\Discstd$ under $\DiagUnten$ in $\DiscDouble$ is only
contained in, but not the whole set,
\begin{equation}
    \label{eq:ExtendedDisc}
    \Discext
    :=
    \set[\big]{
      [p, q] \in \DiscDouble
    }{
      \involutionAll\big([p, q]\big) = [p, q]
    },
\end{equation}
which is a smooth submanifold of $\DiscDouble$ and stable under the
$\group{U}(1, n)$-action. Let
$\DiagUntenExt\colon \Discext \to \DiscDouble$ be the canonical
embedding, which is of course $\group{U}(1, n)$-equivariant.  Then the
diagonal inclusion $\DiagUnten$ of $\Discstd$ in $\DiscDouble$ factors
through $\DiagUntenExt$, such that
$\DiagUnten = \DiagUntenExt \circ \embedInExt$ with a unique
$\group{U}(1, n)$-equivariant smooth map
$\embedInExt\colon \Discstd \to \Discext$. At last, the above
commutative diagram is completed by the smooth
$\group{U}(1, n)$-equivariant map
$\embedInCPnExt\colon \Discext \to \CC\PP^n$,
\begin{equation}
    \label{eq:ExtendedDiscToCPn}
    [p, q]_{\CC_*}
    \mapsto
    \embedInCPnExt\big([p,q]_{\CC_*}\big) := [p]_{\CC_*},
\end{equation}
which is an injective embedding of $\Discext$ in $\CC\PP^n$ with image
$\set{[p] \in \CC\PP^n}{\metric(p) \neq 0}$: This is a consequence of
the observation that, given $p \in \CC^{1+n}$ with
$\metric(p) \neq 0$, there is a unique $q \in \CC^{1+n}$ such that
$[p, q] \in \Discext$, namely $q = -\cc{p}/\metric(p)$, and
$\embedInCPnExt\big([p, q]\big) = [p]$.

%
% The Classical Poisson Algebra
%

\subsection{The Classical Poisson Algebra}
\label{subsection:classically}
The Poisson structure on $\Discstd$ comes from a K\"ahler structure
which can be obtained from a variant of Marsden-Weinstein reduction
from $\CC^{1+n}$: consider the complex $2$-form
$\hOben = h_{\mu\nu} \D \cc{z}^\mu \tensor \D z^\nu \in
\Gamma(\Lambda^2 \Cotangent \CC^{1+n})$,
then $\hOben$ endows $\CC^{1+n}$ with the structure of a pseudo K\"ahler
manifold with pseudo Riemannian metric
$\gOben := \RE(h) = \frac{1}{2} h_{\mu\nu} \D \cc{z}^\mu \vee \D
z^\nu$
and K\"ahler symplectic form
$\omegaOben := \IM(h) = \frac{1}{2\I} h_{\mu\nu}\D \cc{z}^\mu \wedge
\D z^\nu$.
The inverse tensors are then
$\gOben^{-1} = 2 h^{\mu\nu} \partial_{z^\mu}
\vee \partial_{\cc{z}^\nu}$
and $\piOben := \omegaOben^{-1}=2\I h^{\mu\nu}
\partial_{z^\mu}\wedge\partial_{\cc{z}^\nu}$
with $h^{\mu\nu} = h_{\mu\nu}$ for all
$\mu, \nu \in \{0, \ldots, n\}$, so that the Poisson bracket
$\poi{\argument}{\argument}$ on $\CC^{1+n}$ is given by
\begin{equation}
    \label{eq:PoissonBracketOben}
    \poi{a}{b}
    :=
    \piOben\big(\D a \tensor \D b\big)
    =
    2\I h^{\mu\nu}
    \bigg(
    \frac{\partial a}{\partial z^\mu} \frac{\partial b}{\partial \cc{z}^\nu}
    -
    \frac{\partial b}{\partial z^\mu} \frac{\partial a}{\partial \cc{z}^\nu}
    \bigg)
\end{equation}
for all $a, b \in \Smooth(\CC^{1+n})$. It fulfils
$\poi{a}{b}^* = \poi{a^*}{b^*}$ because the Poisson tensor $\pi$ is
real. Note that $\hOben$, hence also $\gOben$ and $\omegaOben$, are
$\group{U}(1, n)$-invariant, so the Poisson bracket is
$\group{U}(1, n)$-equivariant, i.e.
$\poi{a\racts U}{b\racts U} = \poi{a}{b}\racts U$ for all
$a, b \in \Smooth(\CC^{1+n})$ and $U \in \group{U}(1, n)$.

The action of $\group{U}(1, n)$ is not only K\"ahler, i.e. preserves
$\omegaOben$ and $\gOben$, but also Hamiltonian with an equivariant
momentum map. Explicitly, the infinitesimal action of $\lie{u}(1, n)$
on spaces of tensor fields on $\CC^{1+n}$ is given by
$X \racts u = \Lie_{\lieaction{u} + \cc{\lieaction{u}}} X$, where
$\lie{u}(1, n)\ni u \mapsto \lieaction{u} = u^\mu_\nu
z^\nu\partial_{z^\mu}\in \Gamma(\Tangent \CC^{1+n})$
is an anti-morphism of Lie algebras. An equivariant momentum map for
this action is then given by
$\MMup(u)\colon \lie{u}(1, n) \to \Smooth(\CC^{1+n})$,
\begin{equation}
  \MMup(u)
  :=
  \frac{1}{2\I} h_{\mu\nu} u^\mu_\rho z^\rho \cc{z}^\nu,
\end{equation}
because a straightforward calculation shows that
$\poi{f}{\MMup(u)} = \lieaction{u}(f) + \cc{\lieaction{u}}(f) =
f\racts u$
as well as $\poi{\MMup(u)}{\MMup(v)} = \MMup\big(\kom{u}{v}\big)$
holds for all $u, v \in \lie{u}(1, n)$ and $f \in \Smooth(\CC^{1+n})$.

Note that $\metric = \MMup(2\I \Unit_{1+n})$ with
$2\I \Unit_{1+n} \in \lie{u}(1, n)$ generating the
$\group{U}(1)$-subgroup $\set{\E^{\I\phi}\Unit_{1+n}}{\phi \in \RR}$
of $\group{U}(1, n)$, so $\Levelsetstd$ is a $\lie{u}(1)$-level set
and $\Discstd = \Levelsetstd / \group{U}(1)$ the resulting
Marsden-Weinstein quotient. The K\"ahler structure on $\Discstd$ in the
standard coordinates $\chartstd = (w^1, \ldots, w^n)^\Transpose$ is
then given by the $2$-form
\begin{equation}
    \label{eq:hUnten}
    \hUnten
    =
    (1 - \cc{w}\cdot w)^{-2}
    \big(
    \delta_{ij}(1 - \cc{w}\cdot w)
    +
    \cc{w}^k w^\ell \delta_{ki}\delta_{\ell j}
    \big)
    \D \cc{w}^i \tensor \D w^j,
\end{equation}
and thus the K\"ahler metric becomes $\gUnten = \RE(\hUnten)$ while the
K\"ahler symplectic form is $\omegaUnten = \IM(\hUnten)$.  Finally, the
Poisson tensor is
\begin{equation}
    \label{eq:PoissonUnten}
    \piUnten
    =
    \omegaUnten^{-1}
    =
    2\I
    (1 - \cc{w}\cdot w)
    (\delta^{ij} - \cc{w}^iw^j)
    \partial_{w^i} \wedge \partial_{\cc{w}^j}.
\end{equation}

From the point of view of deformation quantization, the description of
the geometric reduction in terms of function algebras becomes more
important: we outline here how this can be accomplished. Let
$\Smooth(\CC^{1+n})_{\Levelsetstd}$ be the $^*$-ideal in
$\Smooth(\CC^{1+n})$ consisting of all functions in
$\Smooth(\CC^{1+n})$ that vanish on $\Levelsetstd$, then
$\Smooth(\CC^{1+n})_{\Levelsetstd}$ is the ideal generated by $\metric+1$ and
$\Smooth(\Levelsetstd) \cong \Smooth(\CC^{1+n}) /
\Smooth(\CC^{1+n})_{\Levelsetstd}$
as $^*$\=/algebras. Moreover, restricted to
$\Smooth(\CC^{1+n})^{\group{U}(1)}$, the Poisson $^*$-subalgebra of
$\Smooth(\CC^{1+n})$ consisting of all $\group{U}(1)$-invariant
functions, the $^*$-ideal $\Smooth(\CC^{1+n})_{\Levelsetstd}$ is even
a Poisson $^*$-ideal, and consequently
$\Smooth(\CC^{1+n})^{\group{U}(1)} /
\big(\Smooth(\CC^{1+n})_{\Levelsetstd} \cap
\Smooth(\CC^{1+n})^{\group{U}(1)}\big)$
is even a Poisson $^*$\=/algebra and isomorphic to the
reduced Poisson $^*$\=/algebra $\Smooth(\Discstd)$. This isomorphism is
described by the classical reduction map:
\begin{definition}[Classical reduction map]
    \label{definition:ClassicalReductionMap}%
    Define the map
    $\ReductionMap{0}\colon \Smooth(\CC^{1+n})^{\group{U}(1)} \to
    \Smooth(\Discstd)$,
    \begin{equation}
        \label{eq:ClassicalReductionMap}
        a\mapsto \ReductionMap{0}(a)
        \quad
        \textrm{with}
        \quad
        \ReductionMap{0}(a)([r]) := a(r)
        \quad
        \textrm{for all}
        \quad
        [r]\in \Discstd.
    \end{equation}
\end{definition}
One can check that $\ReductionMap{0}$ is indeed a well-defined unital
and $\group{U}(1,n)$-equivariant
Poisson $^*$\=/homomorphism. Moreover, its kernel is clearly
$\Smooth(\CC^{1+n})_{\Levelsetstd} \cap
\Smooth(\CC^{1+n})^{\group{U}(1)}$,
and as every smooth $\group{U}(1)$-invariant function on
$\Levelsetstd$ can be extended to a smooth $\group{U}(1)$-invariant
function on $\CC^{1+n}$ (just take an arbitrary extension and make it
$\group{U}(1)$-invariant by averaging over the action of
$\group{U}(1)$), $\ReductionMap{0}$ is also surjective and thus
descends to an isomorphism from the quotient
$\Smooth(\CC^{1+n})^{\group{U}(1)} / \big(
\Smooth(\CC^{1+n})_{\Levelsetstd} \cap
\Smooth(\CC^{1+n})^{\group{U}(1)} \big)$ to $\Smooth(\Discstd)$.

However, as we are interested in non-formal deformation quantizations
of these Poisson algebras, which cannot be constructed on the whole
spaces $\Smooth(\CC^{1+n})$ and $\Smooth(\Discstd)$, we have to
restrict our attention to suitable subalgebras:
\begin{definition}[Spaces of analytic functions on $\CC^{1+n}$ and $\Discstd$]
    \label{definition:AnalyticFunctions}%
    We define
    \begin{equation}
        \label{eq:AnalyticObenUnten}
        \AnalyticOben
        :=
        \set{\hat{a}\circ \DiagOben}
        {\hat{a} \in \Holomorphic(\CC^{1+n}\times\CC^{1+n})}
        \quad
        \textrm{and}
        \quad
        \AnalyticUnten
        :=
        \set{\hat{a}\circ \DiagUnten}
        {\hat{a}\in\Holomorphic(\DiscDouble)}.
    \end{equation}
\end{definition}
Note that these are unital Poisson $^*$-subalgebras of
$\Smooth(\CC^{1+n})$ and $\Smooth(\Discstd)$, respectively, because
the coefficients of the Poisson tensors $\pi$ and $\piUnten$ with
respect to the standard charts are polynomials in the coordinate
functions and because
$(\hat{a}\circ \DiagOben)^* = \cc{\argument} \circ \hat{a} \circ
\involutionAll \circ \DiagOben$
as well as
$(\hat{b}\circ \DiagUnten)^* = \cc{\argument} \circ \hat{b} \circ
\involutionAll \circ \DiagUnten$
holds for all $\hat{a} \in \Holomorphic(\CC^{1+n}\times\CC^{1+n})$ and
$\hat{b} \in \Holomorphic(\DiscDouble)$, where
$\cc{\argument} \circ \hat{a} \circ \involutionAll$ and
$\cc{\argument} \circ \hat{b} \circ \involutionAll$ are compositions
of two anti-holomorphic and one holomorphic function, hence are
holomorphic.  Moreover, the pullback
$\DiagOben^*\colon \Holomorphic(\CC^{1+n} \times \CC^{1+n}) \to
\AnalyticOben$
is an isomorphism of vector spaces, because
$\frac{\partial \hat{a}}{\partial x^\mu} \circ \DiagOben =
\frac{\partial}{\partial z^\mu} (\hat{a} \circ \DiagOben)$
and
$\frac{\partial \hat{a}}{\partial y^\mu} \circ \DiagOben =
\frac{\partial}{\partial \cc{z}^\mu} (\hat{a} \circ \DiagOben)$
holds for all $\mu \in \{0, \ldots, n\}$ and
$\hat{a} \in \Holomorphic(\CC^{1+n} \times \CC^{1+n})$, so
$\hat{a} \circ \DiagOben = 0$ implies that all derivatives of
$\hat{a}$ vanish at all points in the image of $\DiagOben$, and thus
$\hat{a}=0$. Similarly, the pullback
$\DiagUnten^*\colon \Holomorphic(\DiscDouble) \to \AnalyticUnten$ is
an isomorphism as well. The inverse of these isomorphisms will simply
be denoted by $\hat{\argument}$, i.e. given $a \in \AnalyticOben$,
then $\hat{a}$ is the unique element in
$\Holomorphic(\CC^{1+n} \times \CC^{1+n})$ that fulfils
$a = \hat{a} \circ \DiagOben$, similarly for $\AnalyticUnten$. Note
that we have already used this notation for
$\metric = \metricDouble \circ \DiagOben \in \AnalyticOben$ in
\eqref{eq:DefmetricDouble}.

As $\AnalyticOben$ and $\AnalyticUnten$ are isomorphic to spaces of
holomorphic functions, they are Fr\'echet spaces with respect to the
topology of locally uniform convergence of the holomorphic
extensions. We will only need the topology of $\AnalyticUnten$:
\begin{definition}[Norms on $\AnalyticUnten$]
    \label{definition:topAnalyticUnten}%
    Let $K\subseteq \DiscDouble$ be compact, then define the seminorm
    $\seminormUnten{K}{\argument}$ on $\AnalyticUnten$ as
    \begin{equation}
        \label{eq:SupNormK}
        \seminormUnten{K}{a}
        :=
        \sup_{[p, q]\in K} \abs{\hat{a}\big([p, q]\big)}
    \end{equation}
    for all $a \in \AnalyticUnten$.
\end{definition}
It will be helpful to describe $\AnalyticOben$ and $\AnalyticUnten$ as
completions of algebras of polynomials:
\begin{definition}[Polynomials on $\CC^{1+n}$]
    \label{definition:PolynomialOben}%
    For all multiindices $P, Q \in \NN_0^{1+n}$ we define the monomial
    \begin{equation}
        \label{eq:MonomialPQ}
        \monomial{P}{Q}
        :=
        z^P\cc{z}^Q
        :=
        (z^0)^{P_0} \cdots (z^n)^{P_n}
        (\cc{z}^0)^{Q_0} \cdots (\cc{z}^n)^{Q_n}
        \in
        \AnalyticOben
    \end{equation}
    and write $\PolynomeOben$ for their span, i.e. for the space of
    polynomial functions.
\end{definition}
It is then clear that the polynomial functions form a dense unital
Poisson $^*$-subalgebra of $\AnalyticOben$.

Similarly like before, we denote the (closed) subspaces of
$\group{U}(1)$-invariant analytic functions and polynomials on $\CC^{1+n}$ by
$\AnalyticOben^{\group{U}(1)}$ and $\PolynomeOben[\group{U}(1)]$. Then
$\PolynomeOben[\group{U}(1)]$ is spanned by the
$\group{U}(1)$-invariant monomials $\monomial{P}{Q}$ with
$P, Q \in \NN_0^{1+n}$ and $\abs{P} = \abs{Q}$, and is dense in
$\AnalyticOben^{\group{U}(1)}$. This both is a consequence of the
observation that the Cauchy formula for reconstructing the Taylor
coefficient in front of $\monomial{P}{Q}$ by means of circular
integrals around the origin of $\CC^{1+n} \times \CC^{1+n}$ is
$\group{U}(1)$-invariant only if $\abs{P} = \abs{Q}$. Of course,
$\AnalyticOben^{\group{U}(1)}$ and $\PolynomeOben[\group{U}(1)]$ are
again unital Poisson $^*$-subalgebras of $\Smooth(\CC^{1+n})$. Using
the reduction map we can now also construct ``polynomials'' on
$\Discstd$:
\begin{definition}[Polynomials on $\Discstd$]
    \label{definition:PolynomialUnten}%
    We define
    $\monomialDown{P}{Q} := \ReductionMap{0}(\monomial{P}{Q})$ for all
    $P, Q \in \NN_0^{1+n}$ with $\abs{P} = \abs{Q}$ and write
    $\PolynomeUnten$ for the image of $\PolynomeOben[\group{U}(1)]$
    under $\ReductionMap{0}$, i.e. for the span of the functions
    $\monomialDown{P}{Q}$.
\end{definition}
Note that we will show in Theorem~\ref{theorem:Completion} that
$\PolynomeUnten$ is dense in $\AnalyticUnten$ with respect to its
Fr\'echet topology.
As $\ReductionMap{0}$ is not injective, we cannot expect the
monomials $\monomialDown{P}{Q}$ on $\Discstd$ to be a basis of
$\PolynomeUnten$. A suitable choice for a basis is the following
(see \cite[Lemma~4.20]{beiser.waldmann:2014a}):
\begin{definition}[Fundamental monomials on $\Discstd$]
    \label{definition:BasisMonomialsUnten}%
    For all $P, Q \in \NN_0^n$ we define the fundamental monomial
    \begin{equation}
        \monomialDownRed{P}{Q}
        :=
        \begin{cases}
            \monomialDown{(\abs{Q} - \abs{P}, P_1, \ldots, P_n)}
            {(0, Q_1, \ldots, Q_n)}
            &
            \textrm{if }
            \abs{Q} \ge \abs{P}
            \\
            \monomialDown{(0, P_1, \ldots, P_n)}
            {(\abs{P} - \abs{Q}, Q_1, \dots, Q_n)}
            &
            \text{if }
            \abs{Q} \le \abs{P}.
        \end{cases}
    \end{equation}
\end{definition}
Note that, with respect to the coordinates of the standard chart
$\chartstd = (w^1, \ldots, w^n)^\Transpose$, the monomials on
$\Discstd$ are represented as
\begin{equation}
    \label{eq:monomialDownStandard}
    \monomialDown{P}{Q}
    =
    \frac{
      (w^1)^{P_1} \cdots (w^n)^{P_n}
      (\cc{w}^1)^{Q_1} \cdots (\cc{w}^n)^{Q_n}
    }{
      (1 - w \cdot \cc{w})^{\abs{P}}
    }
\end{equation}
for all $P, Q \in \NN_0^{1+n}$ with $\abs{P} = \abs{Q}$. In
particular,
\begin{equation}
    \label{eq:monomialDownRedStandard}
    \monomialDownRed{P}{Q}
    =
    \frac{
      (w^1)^{P_1} \cdots (w^n)^{P_n}
      (\cc{w}^1)^{Q_1} \cdots (\cc{w}^n)^{Q_n}
    }{
      (1 - w \cdot \cc{w})^{\max\{\abs{P}, \abs{Q}\}}
    }
    =
    \frac{
      w^P\cc{w}^Q
    }{
      (1 - w \cdot \cc{w})^{\max\{\abs{P},\abs{Q}\}}
    }
\end{equation}
for all $P, Q \in \NN_0^n$. Using this one can already show that the
$\monomialDownRed{P}{Q}$ are linearly independent, and they span
$\PolynomeUnten$ because every $\monomialDown{P}{Q}$ with
$P, Q \in \NN_0^{1+n}$ and $\abs{P} = \abs{Q}$ can be rewritten as
\begin{equation}
    \label{eq:basisdecompositionUnten}
    \monomialDown{P}{Q}
    =
    \sum_{\substack{T\in\NN_0^n \\ \abs{T}\le \min\{P_0,Q_0\}}}
    \binom{\min\{P_0,Q_0\}}{\abs{T}}
    \frac{\abs{T}!}{T!}
    \monomialDownRed{P'+T}{Q'+T},
\end{equation}
where $P' = (P_1, \ldots, P_n) \in \NN_0^n$ and analogously for $Q$.
So we get (\cite[Lemma~4.20]{beiser.waldmann:2014a}):
\begin{proposition}
    \label{proposition:TheBasis}%
    The fundamental monomials on $\Discstd$ form a basis of
    $\PolynomeUnten$.
\end{proposition}
Note that the product of two fundamental monomials on $\Discstd$ is
more complicated than the product of monomials on $\RR^n$: Given
$\monomialDownRed{P}{Q}, \monomialDownRed{R}{S}$ with
$P, Q, R, S \in \NN_0^n$, then
$\monomialDownRed{P}{Q}\,\monomialDownRed{R}{S} =
\monomialDownRed{P+R}{Q+S}$
holds only in the cases that $\abs{P} \ge \abs{Q}$ and
$\abs{R} \ge \abs{S}$ or that $\abs{P} \le \abs{Q}$ and
$\abs{R} \le \abs{S}$. If $\abs{P} \ge \abs{Q}$ and
$\abs{R} \le \abs{S}$ then
\begin{equation}
    \label{eq:classicalproductUnten}
    \monomialDownRed{P}{Q}\,\monomialDownRed{R}{S}
    =
    \sum_{
      \substack{
        T\in\NN_0^n
        \\
        \abs{T}\le \min\{\abs{S}-\abs{R},\abs{P}-\abs{Q}\}
      }
    }
    \binom{\min\{\abs{S}-\abs{R},\abs{P}-\abs{Q}\}}{\abs{T}}
    \frac{\abs{T}!}{T!}
    \monomialDownRed{P+R+T}{Q+S+T},
\end{equation}
and similarly if $\abs{P} \le \abs{Q}$ and $\abs{R} \ge \abs{S}$. This
especially shows that identifying $\monomialDownRed{P}{Q}$ with a
monomial $z^P\cc{z}^Q$ on $\CC^n$ does not extend to an isomorphism of algebras.

The unital Poisson-$^*$\=/algebras $\PolynomeOben$ and $\PolynomeOben[\group{U}(1)]$
are of course graded by the degree of polynomials. However, the
induced filtration will be even more important in the following,
because it remains well-defined after reduction to $\PolynomeUnten$
and will also be respected by the deformed product:
\begin{definition}[Filtration on polynomials]
    \label{definition:filtration}%
    For all $m \in \NN_0$ we define $\PolynomeOben[\group{U}(1), (m)]$
    as the space of $\group{U}(1)$-invariant polynomials of up to
    degree $2m$, i.e. as the span of $\monomial{P}{Q}$ for all
    $P, Q \in \NN_0$ with $\abs{P} = \abs{Q} \le m$.  Similarly,
    $\PolynomeUnten[(m)]$ is defined as the image of
    $\PolynomeOben[\group{U}(1), (m)]$ under $\ReductionMap{0}$,
    i.e. as the span of $\monomialDown{P}{Q}$ for all $P, Q \in \NN_0$
    with $\abs{P} = \abs{Q} \le m$.
\end{definition}
Similarly to \cite[Lemma~4.18]{beiser.waldmann:2014a}, we get:
\begin{proposition}
    \label{proposition:classicalredkernel}%
    For all $m \in \NN_0$ the following holds:
    \begin{propositionlist}
    \item
        $\dim \PolynomeOben[\group{U}(1), (m)] = \sum_{k=0}^m
        \binom{n+k}{k}^2$.
    \item $\dim \PolynomeUnten[(m)] = \binom{n+m}{m}^2$.
    \end{propositionlist}
    Moreover, the kernel of the restriction of $\ReductionMap{0}$ to
    $\PolynomeOben[\group{U}(1)]$ is the ideal in
    $\PolynomeOben[\group{U}(1)]$ generated by $\metric + 1$, i.e.
    \begin{equation}
        \label{eq:KernelRestrictionMapOnPolynomials}
        \ker \ReductionMap{0} \cap \PolynomeOben[\group{U}(1)]
        =
        \set[\big]{(\metric + 1)a}
        {a \in \PolynomeOben[\group{U}(1)]}.
    \end{equation}
\end{proposition}
\begin{proof}
    Given $k \in \NN_0$ and $\ell \in \NN$ then the set
    $\set[\big]{P \in \NN_0^{\ell}}{\abs{P} = k}$ has
    $\binom{\ell-1+k}{k}$ elements. From this one can easily deduce
    the first dimension formula and also the second because
    \begin{equation*}
        \dim \PolynomeUnten[(m)]
        =
        \bigg(\sum_{k=0}^m \binom{n-1+k}{k}\bigg)^2
        =
        \binom{n+m}{m}^2.
    \end{equation*}
    Moreover, it is easy to see that the $^*$-ideal in
    $\PolynomeOben[\group{U}(1)]$ that is generated by $\metric+1$ is
    in the kernel of $\ReductionMap{0}$ and in order to show that it
    is the whole of
    $\ker \ReductionMap{0} \cap \PolynomeOben[\group{U}(1)]$ it is
    sufficient to show for all $m \in \NN_0$ that
    $\ker \ReductionMap{0} \cap \PolynomeOben[\group{U}(1), (m)] =
    \set{(\metric+1)a}{a \in \PolynomeOben[\group{U}(1), (m)]}$, or
    \begin{equation*}
        \dim \PolynomeOben[\group{U}(1), (m)]
        -
        \dim \PolynomeUnten[(m)]
        \le
        \dim \Big(
        \set[\big]{(\metric+1)a}
        {a\in\PolynomeOben[\group{U}(1)]}
        \cap
        \PolynomeOben[\group{U}(1), (m)]
        \Big)
    \end{equation*}
    Due to the above dimension formulas, the left hand side of this
    reduces to $\sum_{k=0}^{m-1} \binom{n+k}{k}^2$.  In the case that
    $m=0$, this inequality is certainly true. But if it holds for one
    $m \in \NN_0$ then also for $m+1$ because the span of all
    $(\metric + 1) \monomial{P}{Q}$ with $P, Q \in \NN_0$ and
    $\abs{P} = \abs{Q}=m$ has dimension $\binom{n+m}{m}^2$ and is a
    subspace of
    $\set{(\metric+1)a}{a \in \PolynomeOben[\group{U}(1)]} \cap
    \PolynomeOben[\group{U}(1), (m+1)]$
    that has trivial intersection with
    $\PolynomeOben[\group{U}(1), (m)]$.
\end{proof}
So the algebraic description of the reduction stays the same for the
polynnomials, i.e. $\PolynomeUnten$ is isomorphic as a Poisson-$^*$\=/algebra to
the quotient of $\PolynomeOben[\group{U}(1)]$ over the ideal generated by $\metric+1$.

%
% The Deformed Quantum Algebra
%

\subsection{The Deformed Quantum Algebra}
\label{subsection:quantum}
Analogously to the Marsden-Weinstein reduction in the classical case,
the star product on the Poincar\'e disc can be constructed by quantum
reduction: In order to obtain a formal deformation quantisation of
$\Discstd$, one can start with the Wick star product on $\CC^{1+n}$
given by
\begin{equation}
    \label{eq:StarWick}
    a \starWick b
    :=
    \sum_{t=0}^\infty
    \frac{(2\hbar)^t}{t!}
    \sum_{i_1, \ldots, i_t, j_1, \ldots, j_t = 0}^n
    h^{i_1j_1} \cdots h^{i_tj_t}
    \frac{\partial^t a}
    {\partial z^{i_1} \cdots \partial z^{i_t}}
    \frac{\partial^t b}
    {\partial \cc{z}^{j_1} \cdots \partial \cc{z}^{j_t}}
\end{equation}
for all $a, b \in \Smooth(\CC^{1+n})\formalPS{\hbar}$, which is a
$\CC\formalPS{\hbar}$-bilinear associative multiplication and
compatible with the $^*$-involution of pointwise complex conjugation,
i.e.  $(a \starWick b)^* = b^* \starWick a^*$ holds for all
$a, b \in \Smooth(\CC^{1+n})\formalPS{\hbar}$.  Its commutator yields
the classical Poisson bracket up to terms of higher order,
$\frac{\I}{\hbar} \kom{a}{b}_{\starWick} = \poi{a}{b} + \hbar \cdots$.
Note that $\MMup$ is not only a classical moment map but even a
quantum moment map, i.e.
$\frac{\I}{\hbar} \kom{a}{\MMup(u)}_{\starWick} = \poi{a}{\MMup(u)} =
a\racts u$
for all $u \in \lie{u}(1, n)$, because in commutators with
$\MMup(\argument)$, only the first order in $\hbar$ contributes due to
the linearity of the moment map in the $z$- and
$\cc{z}$-coordinates. Analogously to the classical Poisson bracket,
the Wick star product is also $\group{U}(1, n)$-equivariant, i.e.
$(a \racts U) \starWick (b \racts U) = (a \starWick b) \racts U$ holds
for all $a, b \in \Smooth(\CC^{1+n})\formalPS{\hbar}$ and
$U\in\group{U}(1, n)$. The reduced star product algebra on $\Discstd$
can then be obtained from the one on $\CC^{1+n}$ by restriction to the
subalgebra of $\group{U}(1)$-invariant elements in
$\Smooth(\CC^{1+n})\formalPS{\hbar}$ and dividing out the ideal
generated by $\metric + 1$.

On $\PolynomeOben$, the Wick star
product converges trivially for every $\hbar \in \CC$ which yields an
associative product $\starWick[\hbar]$ that fulfils
$(a\starWick[\hbar] b)^* = b^* \starWick[\,\cc{\hbar}\,] a^*$, hence
$\big(\PolynomeOben, \starWick[\hbar], \argument^*\big)$ is a
$^*$\=/algebra for all $\hbar \in \RR$. On the basis of $\PolynomeOben$
given by the monomials $\monomial{P}{Q}$, the Wick star product can be
expressed as
\begin{equation}
  \monomial{P}{Q}\starWick[\hbar]\monomial{R}{S}
  =
  \sum_{T=0}^{\min\{P,S\}}
  (-1)^{T_0} (2\hbar)^{|T|} T! \binom{P}{T}\binom{S}{T}
  \monomial{P+R-T}{Q+S-T}
  \label{eq:starwick}
\end{equation}
for all $P, Q, R, S \in \NN_0^{1+n}$. As the classical $^*$-ideal
generated by $\metric+1$ in $\PolynomeOben[\group{U}(1)]$, i.e. the
kernel of $\ReductionMap{0}$, is no longer an ideal with respect to
the Wick star product, one has to perform an equivalence
transformation first that assures that the star product with $\metric$
is the classical product, see
\cite{bordemann.brischle.emmrich.waldmann:1996a,
  beiser.waldmann:2014a}, and can then restrict to functions on
$\Discstd$. This procedure results in the following deformed reduction
map:
\begin{definition}[Deformed reduction map]
    \label{definition:DeformedRestrictionMap}%
    Let $\HbarDef := \CC_*\backslash\set[]{-1/(2m)}{m \in \NN}$ and
    define for all $\hbar \in \HbarDef$ the deformed reduction map
    $\ReductionMap{\hbar}\colon \PolynomeOben[\group{U}(1)] \to
    \PolynomeUnten$ by linear extension of
    \begin{equation}
        \label{eq:ReductionMapHbar}
        \ReductionMap{\hbar}(\monomial{P}{Q})
        :=
        (2\hbar)^{\abs{P}}
        \bigg(\frac{1}{2\hbar}\bigg)_{\abs{P}}
        \ReductionMap{0}(\monomial{P}{Q})
        =
        (2\hbar)^{\abs{P}}
        \bigg(\frac{1}{2\hbar}\bigg)_{\abs{P}}
        \monomialDown{P}{Q}
    \end{equation}
    for all $P, Q \in \NN_0$ with $\abs{P} = \abs{Q}$, where $(z)_m$
    denotes the Pochhammer symbol, or rising factorial,
    \begin{equation}
        \label{eq:Pochhammer}
        (z)_m := \prod_{k=0}^{m-1} (z+k)
    \end{equation}
    for all $z \in \CC$ and $m \in \NN_0$.
\end{definition}
\begin{proposition}
    \label{proposition:quantumredkernel}%
    For all $\hbar \in \HbarDef$ the kernel of the deformed reduction
    map $\ReductionMap{\hbar}$ is the $^*$-ideal generated by
    $\metric + 1$ with respect to the Wick product $\starWick[\hbar]$
    on $\PolynomeOben[\group{U}(1)]$.
\end{proposition}
\begin{proof}
    Using the explicit formula \eqref{eq:starwick} one can check that
    indeed
    \begin{equation*}
        \monomial{P}{Q}\starWick[\hbar](\metric+1)
        =
        (\metric+1)\starWick[\hbar]\monomial{P}{Q}
        =
        (\metric+1+2\hbar\abs{P})\monomial{P}{Q}
    \end{equation*}
    is in the kernel of $\ReductionMap{\hbar}$ for all
    $P, Q \in \NN_0^{1+n}$ with $\abs{P} = \abs{Q}$, because
    \begin{equation*}
        \ReductionMap{\hbar}
        \big((\metric+1+2\hbar\abs{P})\monomial{P}{Q}\big)
        =
        (2\hbar)^{\abs{P}+1}
        \bigg(\frac{1}{2\hbar}\bigg)_{\abs{P}+1}
        \ReductionMap{0}\big((\metric+1)\monomial{P}{Q}\big)
        =
        0.
    \end{equation*}
    So the $^*$-ideal generated by $\metric+1$ with respect to the
    Wick product on $\PolynomeOben[\group{U}(1)]$ is in the kernel of
    $\ReductionMap{\hbar}$. Conversely, in order to show that this is
    indeed the whole kernel of $\ReductionMap{\hbar}$ one can use the
    same argument as in the proof of
    Proposition~\ref{proposition:classicalredkernel} and count
    dimensions.
\end{proof}
As a consequence, the following product
on $\PolynomeUnten$ is indeed well-defined and associative:
\begin{definition}[Reduced non-formal star product on $\Discstd$]
    \label{definition:starred}%
    For all $\hbar\in\HbarDef$ we define the product
    $\starRed[\hbar]\colon \PolynomeUnten \times \PolynomeUnten \to
    \PolynomeUnten$ as
    \begin{equation}
        a\starRed[\hbar] b := a'\starWick[\hbar] b',
    \end{equation}
    where $a', b' \in \PolynomeOben[\group{U}(1)]$ are arbitrary
    preimages of $a$ and $b$ under $\ReductionMap{\hbar}$.
\end{definition}
Note also that
$(a\starRed[\hbar] b) = a^* \starRed[\,\cc{\hbar}\,] b^*$ holds for
all $\hbar \in \HbarDef$ and $a, b \in \PolynomeUnten$, so pointwise
complex conjugation is a $^*$-involution for $\starRed[\hbar]$ if
$\hbar \in \HbarDef\cap \RR$. An explicit formula for
$\starRed[\hbar]$ on the monomials on $\Discstd$ is
\begin{equation}
  \label{eq:starred}
  \monomialDown{P}{Q} \starRed[\hbar] \monomialDown{R}{S}
  =
  \sum_{T=0}^{\min\{P,S\}}
  (-1)^{T_0}
  \frac{
      (\frac{1}{2\hbar})_{\abs{P+S-T}} T!
    }{
      (\frac{1}{2\hbar})_{\abs{P}}
      (\frac{1}{2\hbar})_{\abs{S}}
  }
  \binom{P}{T}\binom{S}{T}
  \monomialDown{P+R-T}{Q+S-T}
\end{equation}
for all $P, Q, R, S \in \NN_0^{1+n}$ with $\abs{P} = \abs{Q}$ and
$\abs{R} = \abs{S}$. Moreover, as the Wick star product on $\CC^{1+n}$
and the deformes reduction map are $\group{U}(1,n)$-equivariant,
the reduced star product $\starRed[\hbar]$ is also $\group{U}(1,n)$-equivariant.

%
% The Construction of the Fr\'echet Algebra
%

\section{The Construction of the Fr\'echet Algebra}
\label{sec:ConstructionFrechetAlgebra}

In this section we now construct a topology for the
algebra $\PolynomeUnten$ for which the star product becomes
continuous. This will allow to complete the polynomial functions to a
Fr\'echet $^*$\=/algebra. In a second step we show that this (abstract)
completion is still a space of functions on the disc by noting that
the evaluation functionals are continuous. Then the size of this
completion is determined by establishing a bijection to the space
$\AnalyticUnten$, i.e.~those real-analytic functions on the disc
$\Discstd$ which arise as  ``diagonal'' restrictions of holomorphic
functions on $\DiscDouble$.

In \cite[Thm.~4.21, (viii)]{beiser.waldmann:2014a} a topology was
constructed for which the star product is also continuous. However,
the topology we discuss here is slightly coarser which is ultimately
the reason that we are able to determine the completion explicitly in
geometric terms.

%
% The Topology
%

\subsection{The Topology}
\label{subsection:topology}
By constructing a locally convex topology on $\PolynomeUnten$ under
which $\starRed[\hbar]$ is continuous, we can extend the star product
to the completion of $\PolynomeUnten$. A well-behaved topology on
$\PolynomeOben$ has already been examined in \cite{waldmann:2014a,
  schoetz.waldmann:2018a}. Transferring these results to the Poincar\'e
disc is straightforward:
\begin{definition}[Norms on $\PolynomeOben$]
    \label{definition:NormsOben}%
    For all $\rho > 0$ we define the norm
    \begin{equation}
        \label{eq:NormOben}
        \seminormOben[\bigg]{\rho}{
          \sum_{P,Q\in \NN_0^{1+n}}
          a_{P,Q} \monomial{P}{Q}
        }
        :=
        \sum_{P,Q\in \NN_0^{1+n}}
        \abs{a_{P,Q}} \rho^{\abs{P+Q}} \sqrt{\abs{P+Q}!}
    \end{equation}
    on $\PolynomeOben$, where $a_{P,Q} \in \CC$ for all
    $P, Q \in \NN_0^{1+n}$.
\end{definition}
Note that the locally convex topology defined by all these norms
on $\PolynomeOben$ is the same as the one in
\cite{schoetz.waldmann:2018a} (even though we are using a different
fundamental system of continuous seminorms here), if one identifies
$\PolynomeOben$ with the symmetric tensor algebra over
$\CC^{1+n}\times\CC^{1+n}$. Then \cite[Prop.~2.11 and
Lemma~2.12]{schoetz.waldmann:2018a} prove the continuity and absolute
convergence of $\starWick[\hbar]$:
\begin{lemma}
    \label{lemma:wickcont}%
    For every compact $K\subseteq \CC$ and every $\rho > 0$ there
    exist $C, \rho' > 0$ such that the estimate
    \begin{equation}
        \label{eq:ContinuityOfWickOben}
        \seminormOben{\rho}{a\starWick[\hbar]b}
        \le
        \sum_{P,Q,R,S\in\NN_0^{1+n}}
        \abs{a_{P,Q}} \abs{b_{R,S}} \,
        \seminormOben{\rho}{\monomial{P}{Q}\starWick[\hbar]\monomial{R}{S}}
        \le
        C \seminormOben{\rho'}{a} \seminormOben{\rho'}{b}
    \end{equation}
    holds for all $\hbar\in K$ and all
    $a=\sum_{P,Q\in\NN_0^{1+n}} a_{P,Q} \monomial{P}{Q},
      b=\sum_{R,S\in\NN_0^{1+n}} b_{R,S} \monomial{R}{S} \in \PolynomeOben$.
\end{lemma}
It is not very hard to show directly that the estimate in
this lemma holds. However, we will later also need another technical result
about the growth of $\starWick[\hbar]$-powers from
\cite{schoetz.waldmann:2018a}:
\begin{lemma}
    \label{lemma:powergrowthestimate}%
    Let $\hbar \in \RR$ as well as a linear functional
    $\omega\colon \PolynomeOben[\group{U}(1)] \to \CC$ be given, such
    that $\omega$ is continuous with respect to the locally convex
    topology on $\PolynomeOben[\group{U}(1)]$ defined by the norms
    $\seminormOben{\rho}{\argument}$ for all $\rho > 0$ and such that
    $\omega$ is positive with respect $\starWick[\hbar]$, i.e. such
    that $\omega(a^*\starWick[\hbar]a) \ge 0$ holds for all
    $a \in \PolynomeOben[\group{U}(1)]$.  Then for all $k \in \NN_0$
    and all $a \in \PolynomeOben[\group{U}(1), (k)]$ there exist
    $C, D > 0$ with the property that
    \begin{equation}
        \label{eq:OmegaWickPowers}
        \omega\big(
        (a^*)^{\starWick[\hbar] m}
        \starWick[\hbar]
        a^{\starWick[\hbar] m}
        \big)^{\frac{1}{2}}
        \le
        C D^m (km)!
    \end{equation}
    holds for all $m \in \NN_0$, where $a^{\starWick[\hbar] m}$
    denotes the $m$-th power of $a$ with respect to the product
    $\starWick[\hbar]$.
\end{lemma}
\begin{proof}
    Let such $\hbar, \omega,k$ and $a$ be given, then it follows from
    the previous Lemma~\ref{lemma:wickcont} and the continuity of
    $\omega$ that there exists a $C', \rho>0$ with the property that
    $\omega(b^*\starWick[\hbar]b)^{1/2} \le C' \seminormOben{\rho}{b}$
    holds for all $b \in \PolynomeOben[\group{U}(1)]$, and especially
    \begin{equation*}
        \omega\big(
        (a^*)^{\starWick[\hbar] m}
        \starWick[\hbar]
        a^{\starWick[\hbar] m}
        \big)^{\frac{1}{2}}
        \le
        C' \seminormOben{\rho}{a^{\starWick[\hbar] m}}
    \end{equation*}
    for all $m \in \NN_0$. Then
    \cite[Lemma~3.34]{schoetz.waldmann:2018a} shows that there
    exist $C'', D' > 0$ such that
    $\seminormOben{\rho}{a^{\starWick[\hbar] m}} \le
    C''\sqrt{(2km)!}D^{\prime m}$
    holds for all $m \in \NN_0$. As $\sqrt{(2km)!} \le 2^{km} (km)!$,
    this proves the claim with $C = C'C''$ and $D = 2^k D'$.
\end{proof}
Note that we explicitly have to require $\omega$ to be continuous in
this statement: the polynomial functions in $\PolynomeOben$ are
not yet
complete. Hence the usual argument that positive functionals on
Fr\'echet $^*$\=/algebras are automatically continuous
\cite[Thm.~3.6.1]{schmuedgen:1990a}, does not apply here.

Using the explicit basis for $\PolynomeUnten$ we define norms in the
same spirit as before. Note that with our normalization conventions
for the basis functions $\monomialDownRed{P}{Q}$ the following
weighted $\ell^1$-like norms yield a slightly coarser topology than
the original one in \cite{beiser.waldmann:2014a}:
\begin{definition}[Norms on $\PolynomeUnten$]
    \label{definition:NormUnten}%
    For all $\rho > 0$ we define the norm
    \begin{equation}
        \label{eq:NormUnten}
        \seminormUnten[\bigg]{\rho}{
          \sum_{P,Q\in \NN_0^n}
          a_{P,Q} \monomialDownRed{P}{Q}
        }
        :=
        \sum_{P, Q \in \NN_0^n}
        \abs{a_{P,Q}} \rho^{\abs{P+Q}}
    \end{equation}
    on $\PolynomeUnten$, where $a_{P,Q} \in \CC$ for all
    $P,Q\in \NN_0^n$.
\end{definition}

These norms turn out  to be convenient in two essential ways. First,
the practioners will perhaps anticipate  (and we will make this precise  later by
  showing that the completion of $\PolynomeUnten$ w.r.t.~these norms is
  isomorphic to the Fr\'echet space of all \textit{entire} functions on $\CC^n
  \times \CC^n$) that the
  induced topology is exactly the topology of locally uniform convergence on $\Discstd$. Second,
we shall now identify the resulting topology as the quotient topology
with respect to the reduction map $\ReductionMap{\hbar}$! We need the
following well-known estimate for the Pochhammer symbols:
\begin{lemma}
    \label{lemma:Pochhammer}%
    For every compact subset $K \subseteq \CC \backslash (-\NN_0)$
    there exist two constants $\alpha, \omega > 0$ such that
    \begin{equation}
        \label{eq:PochhammerEstiamtes}
        \alpha^m m! \le \abs{(z)_m} \le \omega^m m!
    \end{equation}
    holds for all $z \in K$ and all $m \in \NN_0$.
\end{lemma}
%\begin{proof}
%  By dividing \eqref{eq:PochhammerEstiamtes} by $m!$ one sees that we can
%  choose
%  \begin{equation*}
%    \alpha := \min_{z\in K}\inf_{\ell\in\NN} \abs[\bigg]{\frac{z-1}{\ell}+1}
%    \quad\quad
%    \text{and}
%    \quad\quad
%    \omega := \max_{z\in K}\sup_{\ell\in\NN} \abs[\bigg]{\frac{z-1}{\ell}+1},
%  \end{equation*}
%  both of which exist because
%  $\CC\ni z\mapsto \inf_{\ell\in\NN} \abs{(z-1)/\ell+1} \in \RR$ and
%  $\CC\ni z\mapsto \sup_{\ell\in\NN} \abs{(z-1)/\ell+1} \in RR$, as
%  pointwise infima and suprema of an equicontinuous set of functions,
%  are continuous. Moreover, $\alpha,\omega>0$ holds because
%  $\lim_{\ell\to\infty}\abs{(z-1)/\ell+1}=1$ and $\abs{(z-1)/\ell+1}>0$ for
%  all $z\in K$ and all $\ell\in\NN$ shows that
%  $\inf_{\ell\in\NN} \abs{(z-1)/\ell+1} > 0$ for all $z\in K$.
%\end{proof}
%
With the next lemma we are able to relate the two locally convex
topologies before and after the (quantum) reduction procedure.
\begin{lemma}
    \label{lemma:topred}%
    Let $K \subseteq \HbarDef$ be a compact subset and let $\rho >
    0$. Then there exists a $\rho' > 0$ such that
    \begin{equation}
        \label{eq:ReductionCont}
        \seminormUnten{\rho}{\ReductionMap{\hbar}(a)}
        \le
        \seminormOben{\rho'}{a}
    \end{equation}
    holds for all $\hbar \in K$ and all $a \in \PolynomeOben$.
    Conversely, there exists a $\rho'' > 0$ such that
    \begin{equation}
        \seminormOben{\rho}{\ReductionMapInv{\hbar}(a)}
        \le
        \seminormUnten{\rho''}{a}
    \end{equation}
    holds for all $\hbar \in K$ and all $a \in \PolynomeUnten$, where
    $\ReductionMapInv{\hbar}\colon \PolynomeUnten \to \PolynomeOben$
    is the right inverse of $\ReductionMap{\hbar}$ that is defined as
    the linear extension of
    \begin{equation}
        \ReductionMapInv{\hbar}(\monomialDownRed{P}{Q})
        :=
        \bigg(
        (2\hbar)^{\max\{\abs{P},\abs{Q}\}}
        \bigg(\frac{1}{2\hbar}\bigg)_{\max\{\abs{P},\abs{Q}\}}
        \bigg)^{-1}
        \monomial{\tilde{P}}{\tilde{Q}}
    \end{equation}
    with
    $\tilde{P} := \big( \max\{\abs{Q} - \abs{P}, 0\}, P_1, \ldots,
    P_n\big) \in \NN_0^{1+n}$
    and
    $\tilde{Q} := \big( \max\{\abs{P} - \abs{Q}, 0\}, Q_1, \ldots,
    Q_n\big) \in \NN_0^{1+n}$.
\end{lemma}
\begin{proof}
    Let $K \subseteq \HbarDef$ and, without loss of generality,
    $\rho \ge 1$ be given. Then the previous
    Lemma~\ref{lemma:Pochhammer} shows that there exist
    $\alpha, \omega > 0$ such that
    $\alpha^m m! \le \abs{(1/(2\hbar))_m} \le \omega^m m!$ holds for
    all $m \in \NN_0$ and all $\hbar \in K$.  Define
    $r_{\max} := \max_{\hbar\in K} \abs{2\hbar}$ and
    $r_{\min} := \min_{\hbar\in K} \abs{2\hbar} > 0$.  For all
    $\hbar \in K$ and
    $a = \sum_{P,Q\in \NN_0^{1+n}, \abs{P} = \abs{Q}} a_{P, Q}
    \monomial{P}{Q} \in \PolynomeOben[\group{U}(1)]$
    we get the following estimate with the help of identity
    \eqref{eq:basisdecompositionUnten} and the prime-notation for
    omission of the $0$-component in tuples used there:
    \begin{align*}
        &\seminormUnten{\rho}{\ReductionMap{\hbar}(a)}
        \\
        &\quad=
        \seminormUnten[\bigg]{\rho}{
          \sum_{
            \substack{
              P,Q\in \NN_0^{1+n}\\
              \abs{P} = \abs{Q}
            }
          }
          a_{P,Q} (2\hbar)^{\abs{P}}
          \bigg(\frac{1}{2\hbar}\bigg)_{\abs{P}}
          \monomialDown{P}{Q}
        }
        \\
        &\quad\le
        \sum_{
          \substack{
            P, Q \in \NN_0^{1+n} \\
            \abs{P} = \abs{Q}
          }
        }
        \abs{a_{P,Q}}
        \abs{2\hbar}^{\abs{P}}
        \abs[\bigg]{\bigg(\frac{1}{2\hbar}\bigg)_{\abs{P}}}
        \sum_{
          \substack{
            T \in \NN_0^{n} \\
            \abs{T} \le \min\{P_0, Q_0\}
          }
        }
        \binom{\min\{P_0, Q_0\}}{\abs{T}}
        \frac{\abs{T}!}{T!}
        \rho^{\abs{P'+Q'+2T}}
        \\
        &\quad\le
        \sum_{
          \substack{
            P, Q \in \NN_0^{1+n} \\
            \abs{P} = \abs{Q}
          }
        }
        \abs{a_{P,Q}}
        \Big(\rho\sqrt{\omega r_{\max}}\Big)^{\abs{P+Q}}
        \abs{P}!
        \sum_{
          \substack{
            T \in \NN_0^{n} \\
            \abs{T} \le \min\{P_0, Q_0\}
          }
        }
        \binom{\min\{P_0,Q_0\}}{\abs{T}}
        \frac{\abs{T}!}{T!}
        \\
        &\quad=
        \sum_{
          \substack{
            P, Q \in \NN_0^{1+n} \\
            \abs{P} = \abs{Q}
          }
        }
        \abs{a_{P,Q}}
        \Big(\rho\sqrt{\omega r_{\max}}\Big)^{\abs{P+Q}}
        (1+n)^{\min\{P_0, Q_0\}}
        \abs{P}!
        \\
        &\quad\le
        \sum_{
          \substack{
            P, Q \in \NN_0^{1+n} \\
            \abs{P} = \abs{Q}
          }
        }
        \abs{a_{P,Q}}
        \Big( \rho\sqrt{\omega r_{\max}(1+n)} \Big)^{\abs{P+Q}}
        \sqrt{\abs{P+Q}}!
        \\
        &\quad=
        \seminormOben{\rho\sqrt{\omega r_{\max}(1+n)}}{a}.
    \end{align*}
    This shows the first estimate with
    $\rho' = \rho\sqrt{\omega r_{\max}(1+n)}$.  Conversely,
    for all $\hbar \in K$ and all
    $b = \sum_{P, Q \in \NN_0^n} b_{P,Q} \monomialDownRed{P}{Q} \in
    \PolynomeUnten$ we get
    \begin{align*}
        &\seminormOben{\rho}{\ReductionMapInv{\hbar}(b)}
        \\
        &\quad=
        \seminormOben[\bigg]{\rho}{
          \sum_{P, Q \in \NN_0^n}
          b_{P,Q}
          \bigg(
          (2\hbar)^{\max\{\abs{P},\abs{Q}\}}
          \bigg(\frac{1}{2\hbar}\bigg)_{\max\{\abs{P},\abs{Q}\}}
          \bigg)^{-1}
          \monomial{\tilde{P}}{\tilde{Q}}
        }
        \\
        &\quad\le
        \sum_{P, Q \in \NN_0^n}
        \abs{b_{P,Q}}
        \bigg(
        (r_{\min} \alpha)^{\max\{\abs{P},\abs{Q}\}}
        \big(\max\{\abs{P}, \abs{Q}\}\big)!
        \bigg)^{-1}
        \rho^{2\max\{\abs{P},\abs{Q}\}}
        \sqrt{\big(2\max\{\abs{P},\abs{Q}\}\big)!}
        \\
        &\quad=
        \sum_{P, Q \in \NN_0^n}
        \abs{b_{P,Q}}
        \bigg(\frac{\rho^2}{r_{\min}\alpha}\bigg)^{\max\{\abs{P},\abs{Q}\}}
        \binom{2\max\{\abs{P}, \abs{Q}\}}
        {\max\{\abs{P}, \abs{Q}\}}^{\frac{1}{2}}
        \\
        &\quad\le
        \sum_{P,Q\in \NN_0^n}
        \abs{b_{P,Q}}
        \bigg(\frac{2\rho^2}{r_{\min}\alpha}\bigg)^{\max\{\abs{P},\abs{Q}\}}
        \\
        &\quad\le
        \seminormUnten{\rho''}{b}.
    \end{align*}
    With $\rho'' = \max\{2\rho^2/(r_{\min}\alpha),1\}$ as $\max\{\abs{P},\abs{Q}\} \le \abs{P}+\abs{Q}$.
\end{proof}
The previous Lemma~\ref{lemma:topred} shows that for all
$\hbar \in \HbarDef$ the norms $\seminormUnten{\rho}{\argument}$ with
$\rho > 0$ induce the quotient topology of
$\PolynomeOben[\group{U}(1)]/\ker \ReductionMap{\hbar}$ with the
locally convex topology of the norms $\seminormOben{\rho}{\argument}$
with $\rho > 0$. Together with the continuity of $\starWick[\hbar]$
from Lemma~\ref{lemma:wickcont} this yields:
\begin{theorem}[Continuity of the star product]
    \label{theorem:starRedCont}%
    For every $\hbar\in\HbarDef$ the product $\starRed[\hbar]$ on
    $\PolynomeUnten$ is continuous with respect to the locally convex
    topology defined by the norms $\seminormUnten{\rho}{\argument}$
    for all $\rho > 0$.
\end{theorem}

%
% Characterization of the Completion
%

\subsection{Characterization of the Completion}
Having constructed a suitable locally convex topology on
$\PolynomeUnten$, the next step is to characterize the topology as
well as the completion of the space $\PolynomeUnten$ under this
topology. Understanding various charts on $\DiscDouble$ will be
especially helpful. Recall that we have already defined the standard
chart $\chartstdDouble\colon \DiscDoubleDefStd \to C^\std$ in
\eqref{eq:standardcoord} such that
\begin{equation}
    \label{eq:ChartStdDouble}
    \chartstdDouble \circ \prodiscDouble
    :=
    \left(
        \frac{x^1}{x^0}, \dots, \frac{x^n}{x^0}, \frac{y^1}{y^0}, \dots,
        \frac{y^n}{y^0}
    \right)\bigg|_{\LevelsetDouble},
\end{equation}
where
$\DiscDoubleDefStd = \set[\big]{[p,q] \in \DiscDouble}{x^0(p,q) \neq 0
  \textrm{ and } y^0(p,q) \neq 0}$
and
$C^\std := \set[\big]{(p,q) \in \CC^n \times \CC^n}{p\cdot q \neq 1}$.
We will also need the following two charts:
\begin{definition}[$P$- and $Q$-chart on $\DiscDouble$]
    \label{definition:PQCharts}%
    Let
    $\DiscDoubleDefP := \set[\big]{[p,q] \in \DiscDouble}{y^0(p,q)
      \neq 0}$
    and
    $\DiscDoubleDefQ := \set[\big]{[p,q] \in \DiscDouble}{x^0(p,q)
      \neq 0}$
    and define the $P$-chart
    $\chartPDouble\colon \DiscDoubleDefP \to \CC^n \times \CC^n$ as
    well as the $Q$-chart
    $\chartQDouble\colon \DiscDoubleDefQ \to \CC^n \times \CC^n$ by
    \begin{align}
        \chartPDouble \circ \prodiscDouble
        &:=
        \left(
            x^1 y^0, \ldots, x^n y^0,
            \frac{y^1}{y^0}, \ldots, \frac{y^n}{y^0}
        \right)\bigg|_{\LevelsetDouble}
        \label{eq:Pcoord} \\
        \shortintertext{and}
        \chartQDouble  \circ \prodiscDouble
        &:=
        \left(
            \frac{x^1}{x^0}, \ldots, \frac{x^n}{x^0},
            x^0y^1, \ldots, x^0y^n
        \right)\bigg|_{\LevelsetDouble},
        \label{eq:Qcoord}
    \end{align}
    respectively.
\end{definition}
Note that $\chartstdDouble$, $\chartPDouble$ and $\chartQDouble$ are
all well-defined biholomorphic maps. With respect to these charts, the
monomials $\monomialDownHat{P}{Q}$ with $P, Q \in \NN_0^{1+n}$,
$\abs{P} = \abs{Q}$, are represented as
\begin{align}
    \monomialDownHat{P}{Q} \circ \chartstdDoubleInv
    &=
    \frac{x^{P'}y^{Q'}}{(1 - x \cdot y)^{\abs{P}}} \bigg|_{C^\std},
    \label{eq:chartstddoublerep}
    \\
    \monomialDownHat{P}{Q} \circ \chartPDoubleInv
    &=
    (1 + x \cdot y)^{P_0} x^{P'} y^{Q'},
    \label{eq:chartPdoublerep}
    \\
    \shortintertext{and}
    \monomialDownHat{P}{Q} \circ \chartQDoubleInv
    &=
    (1 + x \cdot y)^{Q_0} x^{P'} y^{Q'},
    \label{eq:chartQdoublerep}
\end{align}
where $P' = (P_1, \ldots, P_n) \in \NN_0^n$ and
$Q' = (Q_1, \ldots, Q_n) \in \NN_0^n$.
In particular, this implies that
$\monomialDownRedHat{P}{Q} \circ \chartPDoubleInv = x^{P} y^{Q}$ for
all $P, Q \in \NN_0^n$ with $\abs{P} \ge \abs{Q}$ and
$\monomialDownRedHat{P}{Q} \circ \chartQDoubleInv = x^{P} y^{Q}$ for
all $P, Q \in \NN_0^n$ with $\abs{P} \le \abs{Q}$, which motivates the
following definition:
\begin{definition}[Coordinate functionals for fundamental monomials]
    \label{definition:CoeffUnten}%
    For all $P, Q \in \NN_0^n$, define the linear functional
    $\CoeffUnten{P}{Q}\colon \AnalyticUnten \to \CC$ as the
    Cauchy integral
    \begin{align}
        \label{eq:CoeffUntenPgeQ}
        \CoeffUnten{P}{Q}(a)
        &:=
        \frac{1}{(-4\pi^2)^n} \oint_{C} \cdots \oint_{C}
        \frac{\hat{a}\circ \chartPDoubleInv}{ x^{P+1} y^{Q+1} }
        \D^n x\wedge \D^n y
        \qquad
        \textrm{if }
        \abs{P}\ge \abs{Q}
        \\
        \label{eq:CoeffUntenQgeP}
        \shortintertext{and}
        \CoeffUnten{P}{Q}(a)
        &:=
        \frac{1}{(-4\pi^2)^n} \oint_{C} \cdots \oint_{C}
        \frac{\hat{a}\circ \chartQDoubleInv}{ x^{P+1} y^{Q+1} }
        \D^n x\wedge \D^n y
        \qquad
        \textrm{if }\abs{P} < \abs{Q}
    \end{align}
    for all $a \in \AnalyticUnten$, where $C\subseteq \CC$
    is a circle around $0$ with arbitrary positive radius and where
    $P+1 := (P_1+1, \ldots, P_n+1)$, analogous for $Q$.
\end{definition}
Using the explicit formulas \eqref{eq:chartPdoublerep} and \eqref{eq:chartQdoublerep}
we immediately get:
\begin{proposition}
    \label{proposition:CoeffUntenDualBasis}%
    For all $P, Q, R, S \in \NN_0^n$, the identity
    \begin{equation}
        \label{eq:CoeffUntenDeltas}
        \CoeffUnten{R}{S}\big(\monomialDownRed{P}{Q}\big)
        =
        \delta_{P,R}\,\delta_{Q,S}
    \end{equation}
    holds.
\end{proposition}
\begin{proposition}
    \label{proposition:CoeffUntenAlternativeFormula}%
    The two formulas for $\CoeffUnten{P}{Q}$ in the $P$- and $Q$-chart
    can be combined into one single formula in the standard-chart,
    namely
    \begin{equation}
        \label{eq:dualbasisstadrep}
        \CoeffUnten{P}{Q}(a)
        =
        \frac{1}{(-4\pi^2)^n} \oint_{C} \cdots \oint_{C}
        \frac{\hat{a} \circ \chartstdDoubleInv}{x^{P+1} y^{Q+1}}
        (1 - x \cdot y)^{\max\{\abs{P}, \abs{Q}\} - 1}
        \D^n x\wedge \D^n y,
    \end{equation}
    for all $a\in\AnalyticUnten$,
    where $C\subseteq \CC$ is a circle around $0$ with radius in
    $]0, 1/\sqrt{n}[$ and where again $P+1 := (P_1+1, \ldots, P_n+1)$,
    analogous for $Q$.
\end{proposition}
\begin{proof}
    The change of coordinates from the standard- to the $P$-chart is
    given by
    \begin{align*}
        \Psi^P
        :=
        \chartPDouble \circ \chartstdDoubleInv
        \colon
        C^\std
        &\to
        \set{(\xi, \eta) \in \CC^n \times \CC^n}
        {\xi \cdot \eta = -1}
        \\
        (\xi, \eta)
        &\mapsto
        \Psi^P(\xi,\eta)
        =
        \bigg(\frac{\xi}{1 - \xi \cdot \eta}, \eta\bigg).
    \end{align*}
    Then
    \begin{align*}
        \CoeffUnten{P}{Q}(a)
        &:=
        \frac{1}{(-4\pi^2)^n} \int_{\Psi^P(C^{2n})}
        \frac{\hat{a} \circ \chartPDoubleInv}{x^{P+1} y^{Q+1}}
        \D^n x\wedge \D^n y \\
        &=
        \frac{1}{(-4\pi^2)^n} \int_{(C')^{2n}}
        \frac{
          \hat{a}\circ \chartPDoubleInv\circ\Psi^P
        }{
          (x^{P+1} y^{Q+1})\circ \Psi^P
        }
        \D^n (x\circ \Psi^P)\wedge \D^n (y\circ\Psi^P) \\
        &=
        \frac{1}{(-4\pi^2)^n} \int_{(C')^{2n}}
        \frac{\hat{a}\circ \chartstdDoubleInv}{ x^{P+1} y^{Q+1} }
        (1 - x \cdot y)^{\abs{P+1}}
        \frac{\D^n x \wedge \D^n y}{(1 - x \cdot y)^{1+n}},
    \end{align*}
    which yields \eqref{eq:dualbasisstadrep} if $\abs{P} \ge \abs{Q}$.
    Note that the calculation of $\D^n (x\circ \Psi^P)\at{\xi,\eta}$
    is easy for $\eta = (1, 0 \ldots, 0)^\Transpose \in \CC^n$, which
    is already sufficient by symmetry. If $\abs{P} < \abs{Q}$, the argument is
    analogous using the $Q$-chart.
\end{proof}

Recall that $\AnalyticUnten$ is endowed with a Fr\'echet topology given
by the seminorms $\seminormUnten{K}{\argument}$ defined for all
compact $K \subseteq \DiscDouble$ in
Definition~\ref{definition:topAnalyticUnten}. We would of course like
to understand the relation between this topology and the topology
defined by the norms $\seminormUnten{\rho}{\argument}$ for all
$\rho > 0$.
\begin{proposition}
    \label{proposition:CoeffCont}%
    For every $P, Q \in \NN_0^n$ the linear functional
    $\CoeffUnten{P}{Q}\colon \AnalyticUnten \to \CC$ is
    continuous. Moreover, for every $\rho > 0$ there exists a compact
    $K\subseteq \DiscDouble$ such that the estimate
    \begin{equation}
        \label{eq:CoeffCont}
        \sum_{P,Q\in \NN_0^n}
        \abs{\CoeffUnten{P}{Q}(a)} \rho^{\abs{P+Q}}
        \le
        2^{2n+2} \seminormUnten{K}{a}
    \end{equation}
    holds for all $a \in \AnalyticUnten$.
\end{proposition}
\begin{proof}
    It is sufficient to show that the estimate \eqref{eq:CoeffCont}
    holds, which is much stronger than mere continuity of
    $\CoeffUnten{P}{Q}$. So let $\rho > 0$ be given and define $K^P$
    and $K^Q$ as the images of the polydiscs with radius $2\rho$ in
    $\CC^n \times \CC^n$ under the holomorphic maps $\chartPDoubleInv$
    and $\chartQDoubleInv$, respectively, and
    $K := K^P \cup K^Q \subseteq \DiscDouble$. Then $K$ is compact and
    from the usual estimate for the Cauchy integral over the boundary
    of a polydisc with radius $2\rho$ it follows for all
    $a \in \PolynomeUnten$ that
    \begin{align*}
        \sum_{P,Q\in \NN_0^n}
        \abs{\CoeffUnten{P}{Q}(a)} \rho^{\abs{P+Q}}
        \le
        \sum_{P,Q\in\NN_0^n}
        \seminormUnten{K}{a}
        \frac{\rho^{\abs{P+Q}}}{(2\rho)^{\abs{P+Q}}}
        =
        2^{2n+2} \seminormUnten{K}{a}.
    \end{align*}
\end{proof}
\begin{lemma}
    \label{lemma:monomialredest}%
    For all compact $K \subseteq \DiscDouble$ there exists a
    $\rho > 0$ such that the estimate
    $\seminormUnten{K}{\monomialDownRed{P}{Q}} \le \rho^{\abs{P+Q}}$
    holds for all $P, Q \in \NN_0^n$.
\end{lemma}
\begin{proof}
    Given such a $K \subseteq \DiscDouble$, then one possible choice
    for $\rho$ is the maximum of
    $\seminormUnten{K}{\monomialDown{E_\mu}{E_\nu}}^2$ over all
    $\mu, \nu \in \{1, \ldots, n\}$.  Submultiplicativity of
    $\seminormUnten{K}{\argument}$ with respect to the pointwise
    product yields
    $\seminormUnten{K}{\monomialDown{P}{Q}} \le
    \sqrt{\rho}^{\abs{P+Q}}$
    for all $P, Q \in \NN_0^{1+n}$ and thus
    $\seminormUnten{K}{\monomialDownRed{P}{Q}} \le
    \sqrt{\rho}^{2\abs{P+Q}} = \rho^{\abs{P+Q}}$
    for all $P, Q \in \NN_0^n$.
\end{proof}
\begin{proposition}
    \label{proposition:PolynomeUntenRvsKTop}%
    On $\PolynomeUnten$ the locally convex topology defined by the
    seminorms $\seminormUnten{\rho}{\argument}$ for all $\rho>0$
    coincides with the subspace topology inherited from
    $\AnalyticUnten$.
\end{proposition}
\begin{proof}
    Let a compact $K \subseteq \DiscDouble$ be given, then by the
    previous Lemma~\ref{lemma:monomialredest} there exists a
    $\rho > 0$ such that
    $\seminormUnten{K}{\monomialDownRed{P}{Q}} \le \rho^{\abs{P+Q}}$
    holds for all $P, Q \in \NN_0^n$, so
    \begin{equation*}
        \seminormUnten{K}{a}
        \le
        \sum_{P, Q \in \NN_0^n}
        \abs{a_{P,Q}}
        \seminormUnten{K}{\monomialDownRed{P}{Q}}
        \le
        \seminormUnten{\rho}{a}
    \end{equation*}
    holds for all
    $a = \sum_{P, Q \in \NN_0^n} a_{P,Q} \monomialDownRed{P}{Q}
    \in\PolynomeUnten$
    with complex coefficients $a_{P,Q}$.  The converse estimate
    follows directly from Proposition~\ref{proposition:CoeffCont},
    which shows that for every $\rho > 0$ there exists a compact
    $K\subseteq \DiscDouble$ such that
    \begin{equation*}
        \seminormUnten{\rho}{a}
        =
        \sum_{P, Q \in \NN_0^n}
        \abs{\CoeffUnten{P}{Q}(a)}\, \rho^{\abs{P+Q}}
        \le
        2^{2n+2}
        \seminormUnten{K}{a}
    \end{equation*}
    holds for all $a \in \PolynomeUnten$.
\end{proof}
\begin{lemma}
    \label{lemma:PolynomeUntenDense}%
    If $a \in \AnalyticUnten$ fulfils
    $\CoeffUnten{P}{Q}(a) = 0$ for all $P, Q \in \NN_0^n$, then
    $a = 0$.
\end{lemma}
\begin{proof}
    Given $a\in\AnalyticUnten$, then
    $\hat{a}\circ \chartPDoubleInv \in \Holomorphic(\CC^n \times
    \CC^n)$,
    so there exist unique complex coefficients $\tilde{a}_{P,Q}$ such
    that
    $\hat{a}\circ \chartPDoubleInv = \sum_{P, Q \in \NN_0^n}
    \tilde{a}_{P,Q} x^P y^Q$
    (and the series converges absolutely and locally uniformly). It is
    sufficient to show that all these coefficients vanish, because the
    domain of the $P$-chart is dense in $\DiscDouble$. From
    the definition of $\CoeffUnten{P}{Q}$ it is immediately clear that
    $\tilde{a}_{P,Q} = 0$ for all $P, Q \in \NN_0^n$ with
    $\abs{P} \ge \abs{Q}$.  Now assume that there is a non-vanishing
    coefficient $\tilde{a}_{P,Q}$, then there is a minimal
    $N \in \NN_0$ such that $\tilde{a}_{P,Q} \neq 0$ for some
    $P, Q \in \NN_0^n$ with $\abs{P} < \abs{Q}$ and $\abs{P+Q} = N$,
    so
    \begin{equation*}
        \hat{a} \circ \chartPDoubleInv
        =
        \sum_{
          \substack{
            P, Q \in \NN_0^n\\
            \abs{P} < \abs{Q}
            \textrm{ and }
            \abs{P+Q}\ge N
          }
        }
        \tilde{a}_{P,Q} x^P y^Q.
    \end{equation*}
    Consider
    $\Psi := \chartPDouble \circ \chartQDoubleInv|_{C^\Psi} \colon
    C^\Psi \to C^\Psi$
    with
    $C^\Psi := \set{(\xi,\eta) \in \CC^{n} \times \CC^{n}}{\xi \cdot
      \eta \neq 1}$,
    which is explicitly given by
    $\Psi(\xi, \eta) = \big(\xi(1 - \xi \cdot \eta), \frac{\eta}{1 -
      \xi \cdot \eta}\big)$
    and describes the change of coordinates between $P$- and
    $Q$-chart. Then
    $\hat{a} \circ \chartPDoubleInv \circ \Psi = \hat{a} \circ
    \chartQDoubleInv|_{C^\Psi}$
    can be represented as the absolutely and locally uniformly
    convergent series
    \begin{align*}
        \hat{a} \circ \chartQDoubleInv|_{C^\Psi}
        =
        \sum_{
          \substack{
            P, Q \in \NN_0^n\\
            \abs{P} < \abs{Q}
            \textrm{ and }
            \abs{P+Q} \ge N
          }
        }
        \tilde{a}_{P,Q}
        \,
        \frac{x^P y^Q}{(1 - x \cdot y)^{\abs{P}-\abs{Q}}}.
    \end{align*}
    It follows that
    $\tilde{a}_{P, Q} = \CoeffUnten{P}{Q}(a)$ for all
    $P, Q \in \NN_0^n$ with $\abs{P} < \abs{Q}$ and $\abs{P+Q} = N$ by
    evaluating the Cauchy-integral for $\CoeffUnten{P}{Q}(a)$ on
    sufficiently small circles in the $Q$-chart. So $\tilde{a}_{P, Q} = 0$ and
    we have a contradiction.
\end{proof}
\begin{theorem}[Completion of $\PolynomeUnten$]
    \label{theorem:Completion}%
    The Fr\'echet $^*$\=/algebra $\AnalyticUnten$ with the pointwise
    operations is the completion of the $^*$\=/algebra $\PolynomeUnten$
    with the pointwise operations and the locally convex topology
    defined by the seminorms $\seminormUnten{\rho}{\argument}$ for all
    $\rho > 0$. Moreover, the functions $\monomialDownRed{P}{Q}$ with
    $P, Q \in \NN_0^n$ form an absolute Schauder basis of
    $\AnalyticUnten$ and the coefficients of the expansion in this
    basis can be calculated explicitly by means of the integral
    formulas for $\CoeffUnten{P}{Q}$ from
    Definition~\ref{definition:CoeffUnten} or
    Proposition~\ref{proposition:CoeffUntenAlternativeFormula}:
    \begin{equation}
        \label{eq:SchauderBasisExpansion}
        a
        =
        \sum_{P, Q \in \NN_0^n}
        \CoeffUnten{P}{Q}(a)
        \,
        \monomialDownRed{P}{Q}
        =
        \sum_{P, Q \in \NN_0^n}
        \frac{\monomialDownRed{P}{Q}}{(-4\pi^2)^n}
        \oint_{C} \cdots \oint_{C}
        \hat{a}
        \,
        \frac{(1 - u \cdot v)^{\max\{\abs{P},\abs{Q}\}-1}}
        {v^{P+1} w^{Q+1}}
        \D^n u\wedge \D^n v,
    \end{equation}
    for all $a \in \AnalyticUnten$, where $u^1,\dots,u^n,v^1,\dots,v^n$ are the
    coordinate functions of the standard chart \eqref{eq:standardcoord}.
\end{theorem}
\begin{proof}
    Proposition~\ref{proposition:PolynomeUntenRvsKTop} shows that the
    $\seminormUnten{\rho}{\argument}$-topology on $\PolynomeUnten$
    coincides with the relative topology inherited from
    $\AnalyticUnten$. Moreover, given $a\in\AnalyticUnten$, then
    $\tilde{a} := \sum_{P, Q \in \NN_0^n} \CoeffUnten{P}{Q}(a) \,
      \monomialDownRed{P}{Q}$
    converges absolutely in $\AnalyticUnten$ due to the
    estimates in Proposition~\ref{proposition:CoeffCont} and
    Lemma~\ref{lemma:monomialredest}. As
    $\CoeffUnten{R}{S}(a)=\CoeffUnten{R}{S}(\tilde{a})$ for all
    $R, S \in \NN_0^n$ due to the continuity of $\CoeffUnten{R}{S}$
    shown in Proposition~\ref{proposition:CoeffCont} and due to the
    identity from Proposition~\ref{proposition:CoeffUntenDualBasis},
    it follows from Lemma~\ref{lemma:PolynomeUntenDense} that
    $a = \tilde{a}$. So $a$ is an element of the closure of
    $\PolynomeUnten$ in $\AnalyticUnten$.  As the functions
    $\monomialDownRed{P}{Q}$ with $P, Q \in \NN_0^n$ are linearly
    independent, this also shows that they form an absolute Schauder
    basis of $\AnalyticUnten$ and that the coefficients of $a$ with
    respect to this basis are the $\CoeffUnten{P}{Q}(a)$.
\end{proof}
Note that we have also shown that $\Holomorphic(\DiscDouble)$ is
isomorphic to $\Holomorphic(\CC^n \times \CC^n)$ as a Fr\'echet
\emph{space} via the isomorphism
\begin{equation}
    \Holomorphic(\DiscDouble)
    \ni
    \sum_{P, Q \in \NN_0^n} a_{P,Q} \monomialDownRedHat{P}{Q}
    \mapsto
    \sum_{P, Q \in \NN_0^n} a_{P,Q} x^P y^Q
    \in
    \Holomorphic(\CC^n \times \CC^n).
\end{equation}
However, this is not an isomorphism of Fr\'echet \emph{algebras} due to
the more complicated formula \eqref{eq:classicalproductUnten} for the
(commutative) product on $\PolynomeUnten$.
\begin{theorem}[The completed star product]
    \label{theorem:TheStarProductAlgebra}%
    For all $\hbar \in \HbarDef$ the product $\starRed[\hbar]$ on
    $\PolynomeUnten$ extends continuously to the completion
    $\AnalyticUnten$, such that $\AnalyticUnten$ with the product
    $\starRed[\hbar]$ becomes a Fr\'echet algebra. The product can explicitly
    be calculated as the series
    \begin{equation}
      a \starRed[\hbar] b
      =
      \sum_{P,Q,R,S \in \NN_0^n}
      a_{P,Q} b_{R,S}\,
      \monomialDownRed{P}{Q} \starRed[\hbar] \monomialDownRed{R}{S}\,,
    \end{equation}
    which converges absolutely and locally uniformly in $\hbar \in \HbarDef$
    for all $a,b\in\AnalyticUnten$ with coefficients $a_{P,Q} := \CoeffUnten{P}{Q}(a)$
    and $b_{R,S} := \CoeffUnten{R}{S}(b)$.
    If $\hbar$ is also
    real, then this Fr\'echet algebra is even a Fr\'echet $^*$\=/algebra
    with pointwise complex conjugation as $^*$-involution.
\end{theorem}
\begin{proof}
    Continuity of $\starRed[\hbar]$ on $\PolynomeUnten$ has already
    been shown in Theorem~\ref{theorem:starRedCont}, so
    $\starRed[\hbar]$ extends continuously to the completion of
    $\PolynomeUnten$, which is $\AnalyticUnten$ by the previous
    Theorem~\ref{theorem:Completion}.

    From the construction of $\starRed[\hbar]$ out of the star product
    $\starWick[\hbar]$ on $\CC^{1+n}$ in Definition~\ref{definition:starred},
    the locally uniform estimate for $\starWick[\hbar]$ in
    Lemma~\ref{lemma:wickcont} and the locally uniform estimates for
    the reduction map $\ReductionMap{\hbar}$ in
    Lemma~\ref{lemma:topred} it follows that the explicit formula for
    $\starRed[\hbar]$ converges absolutely and locally uniformly in
    $\hbar \in \HbarDef$.

    Finally, if $\hbar \in \RR$,
    then pointwise complex conjugation is a $^*$-involution for this
    product by construction, and this also extends to the completion
    by continuity.
\end{proof}

%
% Properties of the Construction
%

\section{Properties of the Construction}
\label{section:properties}

In this section we investigate now some first properties of the
algebra we obtained by completion of the polynomial functions. In
particular, we examine the dependence on $\hbar$ which is now much
more delicate due to the presence of the poles on the negative axis:
unlike for the Weyl algebra on $\CC^{1+n}$ as discussed in
\cite{waldmann:2014a, schoetz.waldmann:2018a} we do not have an entire
deformation anymore. Even worse, the classical limit $\hbar = 0$ is
only a boundary point of the domain where the deformation is
holomorphic. This makes the discussion of the limit $\hbar
\longrightarrow 0$ more involved. In a next step we investigate the
classically positive linear functionals on $\AnalyticUnten$ and show
that the characters of $\AnalyticUnten$, i.e. its (maximal) spectrum,
are given by $\Discext$, finally showing that the completion is indeed
still a space of functions.  Then the deformation $\starRed[\hbar]$
for $\hbar>0$ is shown to have a faithful representation as unbounded
operators on a Hilbert space and we discuss the question whether and
how the infinitesimal action of $\lie{su}(1, n)$ exponentiates to the
global symmetry under the group $\group{SU}(1, n)$. Finally, we
shortly discuss an additional symmetry that occurs in the special case
$n=1$.

%
% Holomorphic Dependence on $\hbar$ and Classical Limit
%

\subsection{\texorpdfstring{Holomorphic Dependence on $\hbar$ and Classical
Limit}{Holomorphic Dependence on h and Classical Limit}}
By now we have seen that for $\hbar = 0$, the completion
of the $^*$\=/algebra $\PolynomeUnten$ is the function algebra $\AnalyticUnten$,
and that the deformed product $\starRed[\hbar]$ extends continuously to
$\AnalyticUnten$ for all $\hbar \in \HbarDef$. We would like to
understand how $a \starRed[\hbar] b$ with $a, b \in \AnalyticUnten$
depends on $\hbar$, especially in the limit $\hbar\to0$.
\begin{theorem}[Holomorphic dependence on $\hbar$]
    \label{theorem:holomorphic}%
    For all $a, b \in \AnalyticUnten$ the function
    \begin{equation}
        \label{eq:HolomorphicInHbar}
        \HbarDef \ni \hbar
        \mapsto
        a \starWick[\hbar] b \in \AnalyticUnten
    \end{equation}
    is holomorphic. The singularities at $\hbar = -1/(2m)$ with
    $m \in \NN$ are at most poles of order $1$.
\end{theorem}
\begin{proof}
    If $a, b \in \PolynomeUnten$, then this all is clear because the
    explicit formula \eqref{eq:starred} for $\starRed[\hbar]$ shows
    that $a \starRed[\hbar] b$ is even a rational function of $\hbar$
    with finitely many poles of at most order $1$ only at the points
    $-1/(2m)$ with $m \in \NN$. From the explicit formula for
    $\starRed[\hbar]$ in Theorem~\ref{theorem:TheStarProductAlgebra}
    and its absolute and locally uniform convergence it follows
    that this result extends to the completion.
\end{proof}
Note that the above theorem does not give any information about the
classical limit $\hbar \to 0$. In fact, this limit is (contrary to the
case of the ordinary Wick star product on $\CC^{1+n}$) non-trivial
because the following example shows that there can indeed occur a pole
at every $\hbar = -1/(2m)$ with $m \in \NN$:
\begin{example}
    \label{example:PolesEverywhere}%
    Let $j,k\in\NN$ be given and write $E_1 := (0,1,0,\dots,0) \in \NN_0^{1+n}$, then
    \begin{align*}
        \monomialDown{j E_1}{j E_1}
        \starRed[\hbar]
        \monomialDown{k E_1}{k E_1}
        &=
        \sum_{t=0}^{\min\{j,k\}}
        \frac{
          (\frac{1}{2\hbar})_{j+k-t}
          t!
        }{
          (\frac{1}{2\hbar})_{j}(\frac{1}{2\hbar})_{k}
        }
        \binom{j}{t}
        \binom{k}{t}
        \monomialDown{(j+k-t)E_1}{(j+k-t)E_1}.\\
        \shortintertext{Moreover,}
        \frac{(z)_{j+k-t}}{(z)_j(z)_k} &= \frac{\prod_{i=\max\{j,k\}}^{j+k-t-1} (z+i)}{\prod_{i=0}^{\min\{j,k\}-1} (z+i)}
    \end{align*}
    has first order poles at all $z = -m$ with $m \in \big\{0, \ldots, \min\{j,k\}-1\big\}$
    and residue
    \begin{equation*}
      (-1)^m \frac{(j+k-t-m-1)!}{m!(j-m-1)!(k-m-1)!},
    \end{equation*}
    whose sign only depends on $m$, but not on $j,k$ or $t$. This implies that
    if $a,b\in \AnalyticUnten$ are of the form
    \begin{equation*}
      a = \sum_{j=0}^\infty a_j \monomialDown{jE_1}{jE_1}
      \quad\quad\text{and}\quad\quad
      b = \sum_{k=0}^\infty b_k \monomialDown{kE_1}{kE_1},
    \end{equation*}
    with positive coefficients $a_j,b_k \in\,\,]0,\infty[$, e.g. $a_j = b_j = 1/j!$, then
    \begin{equation*}
      a\starRed[\hbar] b
      =
      \sum_{j,k=0}^\infty a_j b_k
      \big(
        \monomialDown{j E_1}{j E_1}
        \starRed[\hbar]
        \monomialDown{k E_1}{k E_1}
      \big)
    \end{equation*}
    has simple poles at each of the points $\hbar=-1/(2m)$, $m \in \NN$.
\end{example}
\begin{lemma}
    \label{lemma:classical1}%
    For all $p, s \in \NN_0$ and $x \in [0,1]$, the estimate
    \begin{equation}
        \label{eq:TechnicalEstimate}
        1
        \le
        \frac{
          \prod_{i=0}^{p+s-1}(1+xi)
        }{
          \big(\prod_{j=0}^{p-1}(1+xj)\big) \big(\prod_{k=0}^{s-1}(1+xk)\big)
        }
        \le
        1+ x 2^{p+s}
    \end{equation}
    holds.
\end{lemma}
\begin{proof}
    Without loss of generality we can assume that $p \ge s$. Note that
    \begin{equation*}
        \frac{
          \prod_{i=0}^{p+s-1}(1+xi)
        }{
          \big(\prod_{j=0}^{p-1}(1+xj)\big)
          \big(\prod_{k=0}^{s-1}(1+xk)\big)
        }
        =
        \prod_{k=0}^{s-1} \frac{1+x(p+k)}{1+xk}
    \end{equation*}
    holds. So the first estimate $1 \le \ldots$ is trivial, for
    the second one we will show by induction over $s$ that
    $\prod_{k=0}^{s-1} \frac{1+x(p+k)}{1+xk} \le 1 +
    x\binom{p+s}{s}-x$
    holds. If $s = 0$ or $s = 1$, then this is certainly true, and if
    it holds for one $s \in \NN$ with $s < p$, then also for $s+1$,
    because then
    \begin{align*}
        \prod_{k=0}^{s} \frac{1+x(p+k)}{1+xk}
        &\le
        \bigg(1+x\binom{p+s}{s}-x \bigg) \frac{1+x(p+s)}{1+xs}
        \\
        &=
        1
        +
        x
        \binom{p+s}{s}
        \frac{1+xp+xs}{1+xs}
        -
        x
        +
        \frac{(1-x)xp}{1+xs}
        \\
        &=
        1
        +
        x
        \binom{p+s}{s}
        \frac{1+p+s}{1+s}
        -
        x
        \binom{p+s}{s}
        \frac{(1-x)p}{(1+xs)(1+s)}
        -
        x
        +
        \frac{(1-x)xp}{1+xs}
        \\
        &\le
        1
        +
        x\binom{p+s+1}{s+1}
        -
        x.
    \end{align*}
\end{proof}
\begin{lemma}
    \label{lemma:classical2}%
    For all $t, k_0 \in \NN_0$ and all $x \in [0, 1]$, the estimate
    \begin{equation*}
        \frac{x^t t!}{\prod_{k=k_0}^{k_0+t-1}(1+xk)}
        \le
        x^m 2^t m!
    \end{equation*}
    holds for all $m \in \{0, \ldots, t\}$.
\end{lemma}
\begin{proof}
    As $\frac{1}{1+xk} \le \frac{1}{1+x(k-k_0)}$ it is sufficient to
    prove the estimate for the special case $k_0=0$. If $t=m=0$ then
    this is certainly true, and otherwise $t\ge 1$ and we have
    \begin{equation*}
        \frac{x^t t!}{\prod_{k=0}^{t-1}(1+xk)}
        =
        x^m
        \bigg(\prod_{k=0}^{m-1} \frac{1}{1+xk}\bigg)
        \bigg(\prod_{k=m}^{t-1} \frac{xk}{1+xk}\bigg)
        t (m-1)!
        \le
        x^m t (m-1)!
        \le
        x^m 2^t m!.
    \end{equation*}
\end{proof}
\begin{theorem}[Classical limit]
    \label{theorem:ClassicalLimit}%
    For all $a, b \in \AnalyticUnten$ the functions
    \begin{equation}
        \label{eq:ClassicalLimitProduct}
        ]0,\infty[\,\,\ni \hbar
        \mapsto
        a \starRed[\hbar] b \in \AnalyticUnten
        \quad
        \textrm{and}
        \quad
        ]0,\infty[\,\,\ni \hbar
        \mapsto
        \frac{\I}{\hbar}\kom{a}{b}_{\starRed[\hbar]}
        \in \AnalyticUnten
    \end{equation}
    are continuous and can be extended continuously to $[0,\infty[$ by
    \begin{equation}
        \label{eq:ClassicalLimitTheLimits}
        \lim_{\hbar \to 0^+} a \starRed[\hbar] b
        =
        a b
        \quad
        \textrm{and}
        \quad
        \lim_{\hbar \to 0^+}
        \frac{\I}{\hbar}\kom{a}{b}_{\starRed[\hbar]}
        =
        \poi{a}{b}.
    \end{equation}
\end{theorem}
\begin{proof}
    The continuity of these functions on $]0, \infty[$ is a direct
    consequence of the holomorphic dependence of $\starRed[\hbar]$ on
    $\hbar$ from Theorem~\ref{theorem:holomorphic}.  For the limit
    $\hbar \to 0^+$ we first consider only products of
    $\monomialDown{P}{Q}$ and $\monomialDown{R}{S}$ with
    $P, Q, R, S \in \NN_0^{1+n}$ as well as $\abs{P} = \abs{Q}$ and
    $\abs{R} = \abs{S}$. It will be helpful to use both the
    fundamental system of continuous seminorms
    $\seminormUnten{K}{\argument}$ of $\AnalyticUnten$ with $K$
    running over all compact subsets of $\DiscDouble$ and the
    fundamental system of continuous seminorms
    $\seminormUnten{\rho}{\argument}$ for all $\rho > 0$, extended
    continuously from $\PolynomeUnten$ to $\AnalyticUnten$. Recall
    that these two systems are equivalent by
    Theorem~\ref{theorem:Completion}.  Let $\hbar \in \,\,]0, 1/2]$ and a
    compact $K\subseteq \DiscDouble$ be given, then the estimate
    \begin{align*}
        \seminormUnten[\big]{K}{
          &\monomialDown{P}{Q} \starRed[\hbar] \monomialDown{R}{S}
          -
          \monomialDown{P}{Q} \monomialDown{R}{S}
        }
        \le
        \\
        &\le
        \bigg|
        \frac{
          (\frac{1}{2\hbar})_{\abs{P+S}}
        }{
          (\frac{1}{2\hbar})_{\abs{P}}
          (\frac{1}{2\hbar})_{\abs{S}}
        }
        -
        1
        \bigg|
        \seminormUnten[\big]{K}{\monomialDown{P+R}{Q+S}}
        +
        \sum_{
          \substack{
            T \in \NN_0^{1+n},\\
            \abs{T} > 0
            \text{ and }\\
            T \le \min\{P, S\}
          }
        }
        \frac{
          (\frac{1}{2\hbar})_{\abs{P+S-T}} T!
        }{
          (\frac{1}{2\hbar})_{\abs{P}}
          (\frac{1}{2\hbar})_{\abs{S}}
        }
        \binom{P}{T}\binom{S}{T}
        \seminormUnten[\big]{K}{\monomialDown{P+R-T}{Q+S-T}}
    \end{align*}
    holds by the formula \eqref{eq:starred} for
    $\starRed[\hbar]$. Using the results of the previous two lemmas,
    we get
    \begin{equation*}
        \bigg|
        \frac{
          (\frac{1}{2\hbar})_{\abs{P+S}}
        }{
          (\frac{1}{2\hbar})_{\abs{P}}
          (\frac{1}{2\hbar})_{\abs{S}}
        }
        -
        1
        \bigg|
        =
        \bigg|
        \frac{
          \prod_{i=0}^{\abs{P+S}-1}(1+2\hbar i)
        }{
          \big(\prod_{j=0}^{\abs{P}-1}(1+2\hbar j)\big)
          \big(\prod_{k=0}^{\abs{S}-1}(1+2\hbar k)\big)
        }
        -
        1
        \bigg|
        \le
        2\hbar \, 2^{\abs{P+S}}
    \end{equation*}
    by Lemma~\ref{lemma:classical1} and
    \begin{align*}
        \frac{
          (\frac{1}{2\hbar})_{\abs{P+S-T}} T!
        }{
          (\frac{1}{2\hbar})_{\abs{P}}
          (\frac{1}{2\hbar})_{\abs{S}}
        }
        &=
        \frac{
          \big(\prod_{i=0}^{\abs{P+S-T}-1}(1+2\hbar i)\big)
          (2\hbar)^{\abs{T}} T!
        }{
          \big(\prod_{j=0}^{\abs{P}-1}(1+2\hbar j)\big)
          \big(\prod_{k=0}^{\abs{S-T}-1}(1+2\hbar k)\big)
          \big(\prod_{\ell=\abs{S-T}}^{\abs{S}-1}(1+2\hbar \ell)\big)
        }
        \\
        &\le
        \binom{P+S-T}{P} 2\hbar\, 2^{\abs{T}}
        \\
        &\le
        2\hbar \, 2^{\abs{P+S}}
    \end{align*}
    by Lemma~\ref{lemma:classical2} with $m = 1$ and by using that
    $\frac{1 + 2\hbar i}{1 + 2\hbar k} \le \frac{1 + i}{1 + k}$ as
    long as $i \ge k$.  Now define
    \begin{equation*}
        \rho := 2+\max_{\mu,\nu\in\{1, \ldots, n\}}
        \seminormUnten{K}{\monomialDown{E_\mu}{E_\nu}},
    \end{equation*}
    then
    $\seminormUnten{K}{\monomialDown{P+R-T}{Q+S-T}} \le
    \rho^{\abs{P+R+Q+S-2T}/2} = \rho^{\abs{P+S-T}}$
    by submultiplicativity, and this yields
    \begin{align*}
        \seminormUnten[\big]{K}{
          \monomialDown{P}{Q}
          \starRed[\hbar]
          \monomialDown{R}{S}
          -
          \monomialDown{P}{Q} \monomialDown{R}{S}
        }
        &\le
        2\hbar \, 2^{\abs{P+S}}
        \sum_{T=0}^{\min\{P, S\}}
        \binom{P}{T}\binom{S}{T}
        \rho^{\abs{P+S-T}}
        \\
        &\le
        2\hbar \, (4\rho)^{\abs{P+S}}
        \sum_{T=0}^{\min\{P, S\}}
        \rho^{-\abs{T}}
        \\
        &\le
        2\hbar\,2^{1+n} \, (4\rho)^{\abs{P+S}}
    \end{align*}
    as $\rho\ge2$.
    From the definition of the fundamental monomials it now follows
    that especially
    \begin{equation*}
        \seminormUnten[\big]{K}{
          \monomialDownRed{P}{Q}
          \starRed[\hbar]
          \monomialDownRed{R}{S}
          -
          \monomialDownRed{P}{Q} \monomialDownRed{R}{S}
        }
        \le
        2\hbar\,2^{1+n}
        \,
        (4\rho)^{\max\{\abs{P},\abs{Q}\}+\max\{\abs{R},\abs{S}\}}
    \end{equation*}
    holds for all $P, Q, R, S \in \NN_0^n$. Thus, for all
    $a=\sum_{P, Q \in \NN_0^n} a_{P, Q} \monomialDownRed{P}{Q}$ and
    $b=\sum_{R, S \in \NN_0^n} b_{R, S} \monomialDownRed{R}{S}$ with
    complex coefficients $a_{P, Q}$ and $b_{R, S}$ we get
    \begin{align*}
        \seminormUnten{K}{a\starRed[\hbar] b - ab}
        &\le
        \sum_{P,Q,R,S\in\NN_0^n}
        \abs{a_{P,Q} b_{R,S}}
        \seminormUnten{K}{
          \monomialDownRed{P}{Q}
          \starRed[\hbar]
          \monomialDownRed{R}{S}
          -
          \monomialDownRed{P}{Q} \monomialDownRed{R}{S}
        }
        \\
        &\le
        2\hbar\,2^{1+n}
        \sum_{P,Q,R,S\in\NN_0^n}
        \abs{a_{P,Q} b_{R,S}}
        (4\rho)^{\max\{\abs{P},\abs{Q}\}+\max\{\abs{R},\abs{S}\}}
        \\
        &\le
        2\hbar\,2^{1+n}
        \sum_{P,Q,R,S\in\NN_0^n}
        \abs{a_{P,Q} b_{R,S}}
        (4\rho)^{\abs{P+Q+R+S}}
        \\
        &=
        2\hbar\,2^{1+n}\,
        \seminormUnten{4\rho}{a} \seminormUnten{4\rho}{b},
    \end{align*}
    which proves that
    $\lim_{\hbar \to 0^+} a \starRed[\hbar] b = a b$.  In order to
    prove the result for the limit of the commutator, we proceed
    analogously and start with commutators of $\monomialDown{P}{Q}$
    and $\monomialDown{R}{S}$ with $P, Q, R, S \in \NN_0^{1+n}$. Let
    $\hbar \in \,\,]0, 1/2]$ and a compact $K\subseteq \DiscDouble$ be
    given, then the estimate
    \begin{align*}
        \seminormUnten[\bigg]{K}{
          &\frac{\I}{\hbar}
          \kom{\monomialDown{P}{Q}}
          {\monomialDown{R}{S}}_{\starRed[\hbar]}
          -
          \poi{\monomialDown{P}{Q}}
          {\monomialDown{R}{S}}
        }
        \le
        \\
        &\le
        \frac{1}{\hbar}
        \sum_{m=0}^n
        \bigg|
        \frac{
          (\frac{1}{2\hbar})_{\abs{P+S}-1}
        }{
          (\frac{1}{2\hbar})_{\abs{P}}
          (\frac{1}{2\hbar})_{\abs{S}}
        }
        -
        2\hbar
        \bigg|
        \big( P_m S_m + Q_m R_m \big)
        \seminormUnten[\big]{K}{\monomialDown{P+R-E_m}{Q+S-E_m}}
        \\
        &\quad+
        \frac{1}{\hbar}
        \sum_{
          \substack{
            T \in \NN_0^{1+n},\\
            \abs{T} > 1
            \textrm{ and }
            T \le \min\{P, S\}
          }
        }
        \frac{
          (\frac{1}{2\hbar})_{\abs{P+S-T}} T!
        }{
          (\frac{1}{2\hbar})_{\abs{P}}
          (\frac{1}{2\hbar})_{\abs{S}}
        }
        \bigg(
        \binom{P}{T}\binom{S}{T}
        +
        \binom{Q}{T}\binom{R}{T}
        \bigg)
        \seminormUnten[\big]{K}{\monomialDown{P+R-T}{Q+S-T}}
    \end{align*}
    holds by the formula \eqref{eq:starred} for $\starRed[\hbar]$ and
    because
    \begin{equation*}
        \poi{\monomialDown{P}{Q}}{\monomialDown{R}{S}}
        =
        \lim_{\hbar \to 0^+} \frac{\I}{\hbar}
        \kom{\monomialDown{P}{Q}}
        {\monomialDown{R}{S}}_{\starRed[\hbar]}
        =
        2\I\big( P_m S_m - Q_m R_m \big)
        \monomialDown{P+R-E_m}{Q+S-E_m}
    \end{equation*}
    by construction of $\starRed[\hbar]$.  For the first term we can
    use
    \begin{align*}
        \bigg|
        \frac{
          (\frac{1}{2\hbar})_{\abs{P+S}-1}
        }{
          (\frac{1}{2\hbar})_{\abs{P}}
          (\frac{1}{2\hbar})_{\abs{S}}
        }
        -
        2\hbar
        \bigg|
        &=
        2\hbar
        \bigg|
        \frac{
          (\frac{1}{2\hbar})_{\abs{P+S}}
        }{
          (\frac{1}{2\hbar})_{\abs{P}}
          (\frac{1}{2\hbar})_{\abs{S}}
        }
        \frac{1}{1+2\hbar(\abs{P+S}-1)}
        -
        1
        \bigg|
        \\
        &\le
        2\hbar
        \bigg|
        \frac{
          (\frac{1}{2\hbar})_{\abs{P+S}}
        }{
          (\frac{1}{2\hbar})_{\abs{P}}
          (\frac{1}{2\hbar})_{\abs{S}}
        }
        -1
        \bigg|
        \frac{1}{1+2\hbar(\abs{P+S}-1)}
        +
        \frac{(2\hbar)^2(\abs{P+S}-1)}{1+2\hbar(\abs{P+S}-1)}
        \\
        &\le
        (2\hbar)^2\,\big(2^{\abs{P+S}}+\abs{P+S}\big)
        \\
        &\le
        2\,(2\hbar)^2\,2^{\abs{P+S}}
    \end{align*}
    by Lemma~\ref{lemma:classical1} as long as $\abs{P+S} \ge 1$,
    which is of course the only case of interest. For the second term,
    an analogous argument as before using Lemma~\ref{lemma:classical2}
    with $m = 2$ yields
    \begin{align*}
        \frac{
          (\frac{1}{2\hbar})_{\abs{P+S-T}} T!
        }{
          (\frac{1}{2\hbar})_{\abs{P}}
          (\frac{1}{2\hbar})_{\abs{S}}
        }
        \le
        2\,(2\hbar)^2 \, 2^{\abs{P+S}}
    \end{align*}
    and by putting all of this together we see that
    \begin{align*}
        \seminormUnten[\bigg]{K}{
          \frac{\I}{\hbar}
          \kom{\monomialDown{P}{Q}}{\monomialDown{R}{S}}_{\starRed[\hbar]}
          &-
          \poi{\monomialDown{P}{Q}}{\monomialDown{R}{S}}
        }
        \le
        \\
        &\le
        8\hbar\,2^{\abs{P+S}}
        \sum_{\substack{T\in\NN_0^{1+n},\\ \abs{T}>0\text{ and } T\le\min\{P,S\}}}
        \bigg(
        \binom{P}{T}\binom{S}{T}
        +
        \binom{Q}{T}\binom{R}{T}
        \bigg)
        \rho^{\abs{P+S-T}}
        \\
        &\le
        16\hbar\,(4\rho)^{\abs{P+S}}
        \sum_{\substack{T\in\NN_0^{1+n},\\ \abs{T}>0\text{ and } T\le\min\{P,S\}}}
        \rho^{-\abs{T}}
        \\
        &\le
        16\hbar\,2^{1+n}\,(4\rho)^{\abs{P+S}}
    \end{align*}
    hence
    $\lim_{\hbar\to 0^+} \frac{\I}{\hbar}\kom{a}{b}_{\starRed[\hbar]}
    = \poi{a}{b}$.
\end{proof}

%
% Gel'fand Transformation and Positive Linear Functionals in the
% Classical case
%

\subsection{Gel'fand Transformation and Classically Positive Linear Functionals}
\label{subsec:GelfandTrafoClassical}
From the construction of the commutative $^*$\=/algebra $\AnalyticUnten$
with the pointwise product it is clear that $\AnalyticUnten$ can be
interpreted as a $^*$\=/algebra of functions on
$\Discstd$. Nevertheless, this is to some extend artificial. The most
natural representation of $\AnalyticUnten$ as a $^*$\=/algebra of
functions is on the space of its characters (the unital
$^*$\=/homomorphisms to $\CC$) via Gel'fand transformation. In the
context of topological $^*$\=/algebras, it seems reasonable to focus on
continuous characters. Note, however, that by a theorem of Xia, see
\cite[Thm.~3.6.1]{schmuedgen:1990a}, every positive linear functional on a
unital Fr\'echet-$^*$\=/algebra is actually continuous, so in our case this
is not a restriction.
\begin{proposition} \label{proposition:embedding}
  The map $M\colon \DiscDouble \to \CC^{(1+n)\times(1+n)}$ with components
  $M^{\mu\nu} := \monomialDownHat{E_\mu}{E_\nu}$ is a holomorphic embedding
  that realizes $\DiscDouble$ as the submanifold
  \begin{equation}
    \mathcal{S}
    :=
    \set[\big]{
      A\in\CC^{(1+n)\times(1+n)}
    }{
      h_{\mu\nu}A^{\mu\nu} = -1\text{ and }A^{\mu\nu}A^{\rho\sigma} = A^{\mu\sigma}A^{\rho\nu}
      \text{ for }\mu,\nu,\rho,\sigma \in \{0,\dots,n\}
    }\,.
  \end{equation}
\end{proposition}
\begin{proof}
    First of all we note that $h_{\mu\nu} M^{\mu\nu} =
    -\monomialDownHat{E_0}{E_0} + \sum_{i=1}^n
    \monomialDownHat{E_i}{E_i} = -1$ by
    \eqref{eq:basisdecompositionUnten}, where $E_i$ is the unit vector
    with $1$ at the $i$-th position and $0$ elsewhere. Note also that
    $M^{\mu\nu} \circ \prodiscDouble = (x^\mu y^\nu) \circ
    \inclusionLevelsetDouble$ by construction of the
    $\monomialDownHat{E_\mu}{E_\nu}$, so
  \begin{equation*}
    \big(M^{\mu\nu}M^{\rho\sigma}\big)([p,q])
    =
    x^\mu(p,q)y^\nu(p,q)x^\rho(p,q)y^\sigma(p,q)
    =
    \big(M^{\mu\sigma}M^{\rho\nu}\big)([p,q])
  \end{equation*}
  for all $[p,q]\in\DiscDouble$. As these polynomials are holomorphic, $M$ is
  a holomorphic mapping to $\mathcal{S}$.

  Given $A\in \mathcal{S}$, then $h_{\mu\nu}A^{\mu\nu} = -1$ implies that there exists a
  $\rho\in \{0,\dots,n\}$ such that $A^{\rho\rho} \neq 0$. Define $p,q\in \CC^{1+n}$
  as $p^\mu := A^{\mu\rho}$ and $q^\nu := A^{\rho\nu} / A^{\rho\rho}$,
  then $(p,q)\in\LevelsetDouble$ and $M([p,q]) = A$. Let $\norm{\argument}$ be any
  norm on $\CC^{(1+n)\times(1+n)}$ and consider $B\in \mathcal{S}$ such that
  $\norm{A - B} \le \delta$ for some $\delta\in [0,\infty[$. If $\delta$ is
  sufficiently small, then also $B^{\rho\rho} \neq 0$ and we can construct
  $r,s\in \CC^{1+n}$ as $r^\mu := B^{\mu\rho}$ and $s^\nu := B^{\rho\nu} / B^{\rho\rho}$,
  and again $(r,s)\in\LevelsetDouble$ with $M([r,s]) = B$.
  Now $\abs{p^\mu - r^\mu} = \abs{A^{\mu\rho}-B^{\mu\rho}} \xrightarrow{\delta\to0} 0$
  and $\abs{q^\nu - s^\nu} = \abs{A^{\rho\nu} / A^{\rho\rho}-B^{\rho\nu}/B^{\rho\rho}} \xrightarrow{\delta\to0} 0$
  show that $M$ is injective and a homeomorphism onto its image $\mathcal{S}$.

  Moreover, the $2n+1$ differentials $\D(x^\mu y^\rho), \D(x^\rho y^\nu)$ and
  $\D(x^\rho y^\rho)$ for $\mu,\nu\in\{0,\dots,\rho\}\backslash\{\rho\}$ are
  linearly independent in the point $(p,q)\in \CC^{1+n}\times\CC^{1+n}$. Restricted
  to the submanifold $\LevelsetDouble$ with (complex) codimension $1$, the
  $x^\mu y^\nu$ are the components of $M\circ \prodiscDouble$, which shows that
  the tangent map of $M\circ \prodiscDouble$, hence of $M$, has (at least)
  rank $2n = \dim_\CC \DiscDouble$.
\end{proof}
Note that this especially implies that the holomorphic functions on $\DiscDouble$
separate points, because the holomorphic functions on $\CC^{(1+n)\times(1+n)}$
do.
\begin{definition}
  \label{definition:DeltaFunctional}%
  For all $[p, q] \in \DiscDouble$ we define the evaluation
  functional $\delta_{[p, q]}\colon \AnalyticUnten \to \CC$,
  $a \mapsto \delta_{[p,q]}(a) := \hat{a}([p, q])$.
\end{definition}
\begin{proposition}
  \label{proposition:evaluationfunctionals}%
  Given $[p, q] \in \DiscDouble$, then $\delta_{[p, q]}$ is a
  continuous unital homomorphism from $\AnalyticUnten$ with the
  pointwise product to $\CC$ and it is a continuous character of
  $\AnalyticUnten$ if and only if $[p, q] \in \Discext$.
\end{proposition}
\begin{proof}
  It is immediately clear from its definition that $\delta_{[p, q]}$
  is a continuous unital homomorphism. Moreover,
  $\delta_{[p, q]}(a^*) = \cc{(\hat{a}\circ \involutionAll)([p, q])}$
  and $\cc{\delta_{[p, q]}(a)} = \cc{\hat{a}([p, q])}$, so
  $\delta_{[p, q]}$ is a character if and only if
  $\big(\hat{a} \circ \involutionAll\big)([p, q]) = \hat{a}([p, q])$
  holds for all $a \in \AnalyticUnten$. As the holomorphic functions
  on $\DiscDouble$ separate points by
  Proposition~\ref{proposition:embedding}, this is equivalent to
  $\involutionAll([p, q]) = [p, q]$, i.e. to $[p, q] \in \Discext$.
\end{proof}
\begin{theorem}[Gel'fand transformation]
  \label{theorem:Gelfand}%
  Let $\UnitHom\big(\AnalyticUnten\big)$ be the set of continuous
  unital homomorphisms from $\AnalyticUnten$ to $\CC$ with the
  weak-$^*$-topology.  Then
  $\delta\colon \DiscDouble \to \UnitHom\big(\AnalyticUnten\big)$,
  $[p, q] \mapsto \delta_{[p, q]}$ is a well-defined homeomorphism.
  Moreover, let
  $\Characters\big(\AnalyticUnten\big) \subseteq
  \UnitHom\big(\AnalyticUnten\big)$
  be the set of characters of $\AnalyticUnten$ again with the
  weak-$^*$-topology, then $\delta$ restricts to a homeomorphism
  from $\Discext$ to $\Characters\big(\AnalyticUnten\big)$.
\end{theorem}
\begin{proof}
  Proposition~\ref{proposition:evaluationfunctionals} already shows
  that $\delta$ maps to the continuous unital homomorphisms, and $\delta$ is
  injective because the holomorphic functions on $\DiscDouble$ separate points
  due to Proposition~\ref{proposition:embedding}.

  Now let a continuous unital homomorphism
  $\omega\colon \AnalyticUnten \to \CC$ be given. Construct the
  matrix $A^{\mu\nu} := \omega(\monomialDownHat{E_\mu}{E_\nu})$. Then
  $h_{\mu\nu}A^{\mu\nu} = -\omega\big(\monomialDownHat{E_0}{E_0} -
  \sum_{i=1}^n \monomialDownHat{E_i}{E_i}\big) = -
  \omega(\monomialDownHat{0}{0}) = -1$
  by \eqref{eq:basisdecompositionUnten} and
  $A^{\mu\nu}A^{\rho\sigma}
  = \omega\big( \monomialDown{E_\mu}{E_\nu}\monomialDown{E_\rho}{E_\sigma} \big)
  = \omega\big( \monomialDown{E_\mu}{E_\sigma}\monomialDown{E_\rho}{E_\nu} \big)
  =A^{\mu\sigma}A^{\rho\nu}$, so
  $A$ is in the image of the holomorphic embedding $M$ from
  Proposition~\ref{proposition:embedding} and there exists a unique
  $[p,q] \in \DiscDouble$ with $\omega(\monomialDownHat{E_\mu}{E_\nu}) = A^{\mu\nu} = M^{\mu\nu}([p,q])
  = \delta_{[p,q]} \monomialDownHat{E_\mu}{E_\nu}$ for all $\mu,\nu\in\{0,\dots,n\}$.
  As these monomials generate $\PolynomeUnten$ as a unital $^*$\=/algebra,
  $\delta_{[p, q]}$ and $\omega$ coincide on $\PolynomeUnten$, and as
  $\PolynomeUnten$ is dense in $\AnalyticUnten$ we can conclude that
  $\delta_{[p,q]} = \omega$.

  By now we have seen that $\delta$ is a bijection, and it is even a
  homeomorphism, because the embedding $M$ of $\DiscDouble$ in
  $\CC^{(1+n)\times(1+n)}$ shows that $\DiscDouble$ carries the weak topology
  of its holomorphic functions (because $\CC^{(1+n)\times(1+n)}$ does), which
  under $\delta$ corresponds to the weak-$^*$-topology.

  The analogous statements about the space of characters of
  $\AnalyticUnten$ are now an immediate consequence of the above and
  of Proposition~\ref{proposition:evaluationfunctionals}.
\end{proof}

Note that this result is to some extend unfortunate, as it shows that
the interpretation of $\AnalyticUnten$ as a $^*$\=/algebra of functions
on $\Discstd$ is not really natural. At the center of the problem
lies the fact that the function
$\frac{1}{1 - w \cdot \cc{w}} =
\monomialDown{E_0}{E_0}\in\PolynomeUnten$
is not algebraically positive even though
$\monomialDown{E_0}{E_0} = \ReductionMap{0}\big((\monomial{E_0}{0})^*
\monomial{E_0}{0}\big)$.
As the algebra $\PolynomeUnten$ arises from $\PolynomeOben$ by a
reduction procedure (in the classical case as well as in the quantum
case), one might consider only those positive linear functionals on
$\PolynomeUnten$ to be ``relevant'' that come from a
$\group{U}(1)$-invariant functional on $\PolynomeOben$. This would
especially eliminate all characters
$\phi \in \Characters\big(\AnalyticUnten\big)$ for which
$\phi(\monomialDown{E_0}{E_0}) < 0$ and leave only the evaluation
functionals at points in $\Discstd$.

\begin{corollary}
    \label{corollary:RadonStuff}%
    Let $\phi\colon \AnalyticUnten \to \CC$ be a continuous positive
    linear functional (with respect to the pointwise product), then
    there exists a compact $K\subseteq \Discext$ and a Radon measure
    $\mu$ on $K$ such that
    \begin{align*}
        \phi(a) = \int_K \hat{a} \D \mu
    \end{align*}
    holds for all $a \in \AnalyticUnten$.
\end{corollary}
\begin{proof}
    It is sufficient to treat the case that $\phi$ is normalized to
    $\phi(\Unit) = 1$, as $\phi(\Unit) = 0$ implies $\phi = 0$ by the
    Cauchy-Schwarz inequality, and to show that there exists a compact
    $K\subseteq \Discext$ such that $\abs{\phi(a)} \le \seminormUnten{K}{a}$
    holds for all $a\in \AnalyticUnten$, in which case $\phi$ extends
    continuously to $\Stetig(K)$, the completion of $\AnalyticUnten$ under
    $\seminormUnten{K}{\argument}$ by the Stone-Weierstra{\ss} theorem, and can be
    represented by integration over a Radon measure $\mu$ on $K$ by the
    Riesz-Markov theorem.

    As $\phi$ is
    continuous, there exists a compact $K' \subseteq \DiscDouble$,
    stable under $\involutionAll$, and a constant $C \in \RR$ such
    that $\abs{\phi(a)} \le C \seminormUnten{K'}{a}$ holds for all
    $a\in\AnalyticUnten$.  From the submultiplicativity of
    $\seminormUnten{K'}{\argument}$ and the Cauchy-Schwarz
    inequality it then follows that
    $\seminormUnten{\phi, \infty}{a} := \sup_{b \in \AnalyticUnten,
      \phi(b^*b) = 1} \sqrt{\phi(b^*a^*ab)}$
    is a continuous seminorm on $\AnalyticUnten$, because
    \begin{align*}
        \sqrt{\phi(b^*a^*ab)}
        &\le
        \sqrt[4]{\phi(b^*(a^*a)^2b)}
        \\
        &\le
        \dots
        \\
        &\le
        \sqrt[2^{n}]{\phi(b^*(a^*a)^{2^{n-1}}b)}
        \\
        &\le
        \sqrt[2^{n}]{C \seminormUnten{K'}{b^*} \seminormUnten{K'}{b} }
        \sqrt{\seminormUnten{K'}{a^*}\seminormUnten{K'}{a}}
    \end{align*}
    holds for all $n \in \NN$ and all $b \in \AnalyticUnten$ with
    $\phi(b^*b) = 1$, hence
    $\sqrt{\phi(b^*a^*ab)} \le \seminormUnten{K'}{a}$ and
    $\seminormUnten{\phi,\infty}{\argument} \le
    \seminormUnten{K'}{\argument}$.
    One can also check that $\seminormUnten{\phi, \infty}{\argument}$
    is a $C^*$-seminorm (it is the operator norm in the usual GNS representation
    associated to $\phi$). By dividing out the zeros of
    $\seminormUnten{\phi, \infty}{\argument}$ and completing with respect
    to $\seminormUnten{\phi, \infty}{\argument}$, we construct a commutative
    $C^*$-algebra $\algebra{B}$ and the continuous
    $\iota\colon \AnalyticUnten \to \algebra{B}$ as the composition of the
    projection on the quotient and the inclusion in the completion. Let
    $\Spec^*(\algebra{B})$ be the (compact) set of characters of $\algebra{B}$, then the
    $C^*$-norm $\seminormUnten{\phi,\infty}{\argument}$ on $\algebra{B}$ is
    the uniform norm on the Gel'fand transformation of $\algebra{B}$, hence
    especially $\seminormUnten{\phi,\infty}{a} =
    \sup_{\psi\in\Spec^*(\algebra{B})} \abs{\psi(\iota(a))}$ for all $a\in\AnalyticUnten$.
    The pullback $\iota^*\colon \Spec^*(\algebra{B})\to\Spec^*(\AnalyticUnten)$ is
    weak-$^*$-continuous by construction of $\iota$ and by the previous
    Theorem~\ref{theorem:Gelfand}, the compact
    $\iota^*\big(\Spec^*(\algebra{B})\big) \subseteq \Spec^*(\AnalyticUnten)$
    is the image of a compact $K\subseteq \Discext$ under $\delta$, so
    \begin{equation*}
      \abs{\phi(a)}
      \le
      \phi(a^*a)
      \le
      \seminormUnten{\phi,\infty}{a}
      =
      \sup_{\psi\in\Spec^*(\algebra{B})} \abs[\big]{\psi\big(\iota(a)\big)}
      =
      \sup_{[p,q]\in K} \abs{\delta_{[p,q]}a}
      =
      \sup_{[p,q]\in K} \abs{\hat{a}([p,q])}
      =
      \seminormUnten{K}{a}
    \end{equation*}
    for all $a\in \AnalyticUnten$.
\end{proof}

%
% Positive Linear Functionals and Representations of the Deformed Algebra
%

\subsection{Positive Linear Functionals and Representations of the Deformed
Algebra}
\label{subsec:PositiveFunctionalsDeformed}
From the point of view of physics, the most important problem after
having constructed a $^*$\=/algebra of observables is, whether there
exist many positive linear functionals and thus faithful
representations. Again, we focus on continuous positive linear
functionals as we are dealing with a locally convex $^*$\=/algebras,
which is no restriction as mentioned before, because every positive
linear functional on a Fr\'echet-$^*$\=/algebra is continuous.
\begin{proposition}
    \label{proposition:evfuncpos}%
    For every $[r]\in \Discstd$, the evaluation functional
    $\delta_{[r]} := \delta_{\DiagUnten([r])} \colon \AnalyticUnten
    \to \CC$,
    $a \mapsto \delta_{[r]}(a) = a([r])$ is continuous and positive
    with respect to every product $\starRed[\hbar]$ for all
    $\hbar \ge 0$, i.e. $\delta_{[r]}(a^* \starRed[\hbar] a) \ge 0$
    holds for all $a \in \AnalyticUnten$.
\end{proposition}
\begin{proof}
    Like in Proposition~\ref{proposition:evaluationfunctionals},
    continuity of all evaluation functionals $\delta_{[r]}$ with
    $[r]\in\Discstd$ is clear. It is thus also sufficient to prove
    positivity of $\delta_{[r]}$ only on the dense unital $^*$-subalgebra
    $\PolynomeUnten$ of $\AnalyticUnten$, which has already been
    done in \cite[Lemma~5.21]{beiser.waldmann:2014a}.
\end{proof}
\begin{corollary}
    \label{corollary:StronglyPositiveDeformation}%
    Let $\hbar \ge 0$. Then every positive linear functional on
    $\Cinfty(\Discstd)$ restricts to a positive linear functional on
    $\AnalyticUnten$ with respect to $\starRed[\hbar]$.
\end{corollary}
\begin{proof}
    Indeed, every such positive linear functional is an integration
    with respect to a Radon measure on a compact subset $K \subseteq
    \Discstd$. Note that now the support $K$ is contained in
    $\Discstd$ instead of $\Discext$, see
    Corollary~\ref{corollary:RadonStuff}. Since all evaluation
    functionals at points of $\Discstd$ are positive with respect to
    $\starRed[\hbar]$, this also holds for the convex combinations
    needed for general Radon measures.
\end{proof}

Given a pre-Hilbert space $\PreHilb$ with inner product
$\skal{\argument}{\argument}_{\PreHilb}$, then we write
$\Adbar(\PreHilb)$ for the unital $^*$\=/algebra of all adjointable
endomorphisms of $\PreHilb$, where being adjointable is to be
understood in the purely algebraic sense that for an endomorphism
$a\in\Adbar(\PreHilb)$ there exists a (necessarily unique)
$a^*\in\Adbar(\PreHilb)$ such that
$\skal{\phi}{a(\psi)}_{\PreHilb} = \skal{a^*(\phi)}{\psi}_{\PreHilb}$
holds for all $\phi,\psi\in\PreHilb$.
\begin{definition}[Continuous representation]
    \label{definition:ContinuousRepresentation}%
    Let $\hbar \ge 0$. A continuous representation of the
    Fr\'echet-$^*$\=/algebra
    $\big(\AnalyticUnten,\starRed[\hbar], \argument^*\big)$ is defined
    as a tuple $(\PreHilb, \rep)$ consisting of a pre-Hilbert space
    $\PreHilb$ and a unital $^*$\=/homomorphism
    $\rep\colon \AnalyticUnten \to \Adbar(\PreHilb)$ that is
    continuous with respect to the weak topology on
    $\Adbar(\PreHilb)$, i.e. the topology defined by the seminorms
    $\Adbar(\PreHilb)\ni a \mapsto
    \abs{\skal{\phi}{a(\psi)}_{\PreHilb}}$
    for all $\phi, \psi \in \PreHilb$.
\end{definition}

The following is well-known: As the product $\starRed[\hbar]$ on
$\AnalyticUnten$ is continuous, one could equivalently demand that
$\rep$ be continuous with respect to the strong topology on
$\Adbar(\PreHilb)$, i.e. the topology defined by the seminorms
$\Adbar(\PreHilb)\ni a \mapsto
\sqrt{\skal{a(\phi)}{a(\phi)}_{\PreHilb}}$
for all $\phi \in \PreHilb$, because
$\AnalyticUnten \ni a \mapsto
\sqrt{\skal{\rep(a)(\phi)}{\rep(a)(\phi)}_{\PreHilb}} =
\sqrt{\skal{\phi}{\rep(a^*\starRed[\hbar]a)(\phi)}_{\PreHilb}} \in
\RR$
is a continuous seminorm on $\AnalyticUnten$ for all
$\phi \in \PreHilb$ if $\rep$ is weakly continuous. Moreover, there
exists a faithful continuous representation of
$\big(\AnalyticUnten, \starRed[\hbar], \argument^*\big)$ if and only
if the continuous positive linear functionals on
$\big(\AnalyticUnten, \starRed[\hbar], \argument^*\big)$ separate
points, i.e. if and only if $\rho(a) = 0$ for one
$a \in \AnalyticUnten$ and all continuous linear functionals
$\rho\colon \AnalyticUnten \to \CC$ that are positive with respect to
$\starRed[\hbar]$ implies that $a = 0$. This is because, on the one
hand, given such a continuous representation, then every vector
$\phi \in \PreHilb$ yields a weakly continuous positive linear
functional
$\Adbar(\PreHilb)\ni a \mapsto \skal{\phi}{a(\phi)}_{\PreHilb}\in
\CC$,
that can be pulled back to a continuous positive linear functional on
$\AnalyticUnten$ by $\rep$, and on the other, the GNS-construction
allows to construct continuous representations out of continuous
positive linear functionals. In this case, a $^*$\=/algebra is also
called $^*$-semisimple \cite[Def.~6.4.1]{schmuedgen:1990a}. In our
case, this allows for the following conclusion:
\begin{theorem}[Existence of faithful continuous representations]
    \label{theorem:FaithfulRep}%
    Let $\hbar \ge 0$ be given, then there exists a faithful
    continuous representation of the Fr\'echet $^*$\=/algebra
    $\big(\AnalyticUnten, \starRed[\hbar], \argument^*\big)$.
\end{theorem}
\begin{proof}
    The evaluation functionals $\delta_{[r]}$ for all
    $[r] \in \Discstd$ from the previous
    Proposition~\ref{proposition:evfuncpos} are continuous and
    positive and clearly separate points.
\end{proof}

Having established the existence of interesting representations by
(unbounded) operators, the question arises which algebra elements are
actually essentially self-adjoint in representations. Therefore,
recall Definition~\ref{definition:filtration} of the filtration of
$\PolynomeUnten$ by degree and Lemma~\ref{lemma:powergrowthestimate}
for the estimate on the growth of $\starRed[\hbar]$-powers. Note that
the formula \eqref{eq:starred} also shows immediately that the
deformed products $\starRed[\hbar]$ are also filtered with respect to
the above filtration, i.e.
$a \starRed[\hbar] b \in \PolynomeUnten[(k+\ell)]$ holds for all
$a \in \PolynomeUnten[(k)], b \in \PolynomeUnten[(\ell)]$ and all
$\hbar \in \HbarDef$.
\begin{theorem}[Essential self-adjointness of observables]
    \label{theorem:selfadjoint}%
    Fix $\hbar \ge 0$ and let $(\PreHilb, \pi)$ be a continuous
    $^*$\=/representation of
    $\big(\AnalyticUnten, \starRed[\hbar], \argument^*\big)$. Then
    $\pi(a)$ is essentially self-adjoint for every Hermitian
    $a \in \PolynomeUnten[(1)]$ and for every Hermitian
    $a \in \PolynomeUnten[(2)]$ that is semi-bounded, i.e. for which
    the set of all $\phi(a) / \phi(\Unit)$ with $\phi$ running over all non-zero
    continuous positive linear functionals is bounded from above or below.
\end{theorem}
\begin{proof}
    If $a$ is Hermitian then $\pi(a)$ is symmetric. Moreover, every
    vector $\phi \in \PreHilb$ yields a continuous positive linear
    functional
    $\AnalyticUnten\ni a \mapsto \skal{\phi}{\pi(a)(\phi)}_\PreHilb
    \in \CC$
    on $\AnalyticUnten$, which can be pulled back to a continuous
    positive linear functional on
    $\PolynomeOben[\group{U}(1)] \subseteq \PolynomeOben$
    with the Wick star product $\starWick[\hbar]$.
    The estimate from Lemma~\ref{lemma:powergrowthestimate} then
    shows that every such
    $\phi \in \PreHilb$ is an analytic vector of every $\pi(a)$ if
    $a \in \PolynomeUnten[(1)]$, and a semi-analytic vector of every
    $\pi(a)$ if $a \in \PolynomeUnten[(2)]$.  In both cases it follows
    from Nelson's criterium for self-adjointness that $\pi(a)$ is essentially
    self-adjoint, see e.g. \cite{masson.mcclary:1972a} for a direct proof.
\end{proof}

%
% Exponentiation of $\group{su}(1, n)$ action
%

\subsection{\texorpdfstring{Exponentiation of the
$\group{su}(1, n)$-Action}{Exponentiation of the su(1,n)-Action}}
\label{subsec:ExponentationSUAction}

By reduction to $\Discstd$, the $\group{U}(1, n)$-symmetry of
$\CC^{1+n}$ is reduced to a $\group{SU}(1, n)$-symmetry. Recall that
$\lie{u}(1, n) \ni u \mapsto \MMup(u) := \frac{1}{2\I} h_{\mu\nu}
u^\mu_\rho z^\rho \cc{z}^\nu \in \PolynomeOben \subseteq
\Smooth(\CC^{1+n})$
is a (classical) equivariant moment map for this action. As $\MMup(u)$
is linear in the $z$ and $\cc{z}$-coordinates, only terms up to first
order in $\hbar$ will contribute to the Wick star product with
$\MMup(u)$, so
$f \racts u = \poi{f}{\MMup(u)} = \kom{f}{\frac{\I}{\hbar}\MMup(u)}$
and
$\frac{\I}{\hbar} \MMup(\kom{u}{v}) =
\frac{\I}{\hbar}\poi{\MMup(u)}{\MMup(v)} =
\kom{\frac{\I}{\hbar}\MMup(u)}{\frac{\I}{\hbar}\MMup(v)}$
hold for all $u, v \in \lie{u}(1, n)$ and $f \in \PolynomeOben$,
i.e. $\frac{\I}{\hbar} \MMup$ is an equivariant quantum moment
map. Reduction to $\Discstd$ then yields the following well-known
result, see
e.g. \cite[Lemma~5]{bordemann.brischle.emmrich.waldmann:1996b} for the
case of reduction to $\mathbbm{CP}^n$:
\begin{proposition}
    \label{proposition:QuantumMomentumMap}%
    The map $\MMdown\colon \lie{su}(1, n) \to \PolynomeUnten$,
    $u \mapsto \MMdown(u) := \big(\ReductionMap{0}\circ\MMup\big)(u)= \frac{1}{2\I} h_{\mu\nu} u^\mu_\rho
    \monomialDown{E_\rho}{E_\nu}$
    is a classical equivariant moment map (with respect to the Poisson
    tensor $\piUnten$ on $\Discstd$) and
    $\ReductionMap{\hbar}\circ\frac{\I}{\hbar}\MMup = \frac{\I}{\hbar} \MMdown$ an
    equivariant quantum moment map (with respect to $\starRed[\hbar]$)
    for all $\hbar \in \HbarDef$.
\end{proposition}
\begin{proof}
    This follows directly from $\ReductionMap{0}$ and $\ReductionMap{\hbar}$
    being $\group{U}(1,n)$-equivariant and the algebraic version of the
    reduction procedure, i.e. that $\ReductionMap{0}$ is a morphism of
    Poisson-$^*$\=/algebras, or the construction of $\starRed[\hbar]$
    such that $\ReductionMap{\hbar}$ becomes a morphism of
    $^*$\=/algebras, respectively.
\end{proof}

It would of course be a nice property of the deformed algebra if we
could exponentiate the inner action of the $\lie{su}(1, n)$-algebra to
an inner action of the $\group{SU}(1, n)$-group. So note that the
image of $\MMdown$ is in $\PolynomeUnten[(1)]$. However, from
Lemma~\ref{lemma:powergrowthestimate} one cannot deduce that the
$\starRed[\hbar]$-exponential series of all elements
$a \in \PolynomeUnten[(1)]$ converges. In fact:
\begin{example}
    Let $n = 1$, $\hbar = 1/2$ and
    $a = \monomialDown{E_0}{E_1} = \monomialDownRed{0}{1}$, then the
    $m$-th $\starRed[\hbar]$-power of $a$ is
    \begin{equation}
        \label{eq:ExponentiateMMap}
        a^{\starRed[\hbar] m}
        =
        m!\,\monomialDown{mE_0}{mE_1}
        =
        m!\,\monomialDownRed{0}{m},
    \end{equation}
    and so
    $\lim_{M \to \infty} \sum_{m = 0}^M a^{\starRed[\hbar] m} / m!  =
    \lim_{M \to \infty} \sum_{m=0}^M \monomialDownRed{0}{m}$
    is not a Cauchy-sequence in any topology on $\PolynomeUnten$ that
    makes the evaluation functionals $\delta_{[r]}$ at all
    $[r]\in\Discstd$ continuous, as this series does not converge
    in the point $[r] \in \Discstd$ with $\chartstd([r]) = w^1([r]) = 1/\sqrt{2}$,
    where $\monomialDownRed{0}{m}([r]) = 2^{m/2}$.
\end{example}
Note that this also rules out the existence of any locally
\emph{multiplicatively} convex topology on $\PolynomeUnten$ that makes
all these evaluation functionals continuous: the example shows that
there is no entire calculus which for a locally multiplicatively
convex algebra would exist.

Nevertheless, by Theorem~\ref{theorem:selfadjoint}, all elements in
the image of $\MMdown$ are essentially self-adjoint in every
continuous $^*$\=/representation for all $\hbar > 0$. Moreover,
Nelson's theorem even allows to exponentiate this inner Lie algebra action
to an inner Lie group action in such representations:
\begin{theorem}[Exponentiation of $\lie{su}(1, n)$-action]
    \label{theorem:ExponentiateStuff}%
    Fix $\hbar > 0$ and let $(\PreHilb, \pi)$ be a continuous
    $^*$\=/representation of
    $\big(\AnalyticUnten, \starRed[\hbar], \argument^*\big)$ and
    $\PreHilb^{\textup{cpl}}$ the completion of $\PreHilb$. Then there
    exists a unique unitary representation
    $\GroupRep\colon \group{SU}(1, n) \to
    \Unitary(\PreHilb^{\textup{cpl}})$
    such that
    $\pi(\MMdown(u))^\OpClosure =
    \GroupRepDeriv(u)^\OpClosure$
    holds for all $u \in \lie{su}(1, n)$, where $\argument^\OpClosure$
    denotes the closure of an operator on $\PreHilb^{\textup{cpl}}$ and
    $\GroupRepDeriv(u)$ the derivation of the representation
    $\GroupRep$ at the neutral element in direction $u$, i.e. $\GroupRepDeriv(u)$
    is the
    operator in $\PreHilb^{\textup{cpl}}$ whose domain is
    $\Smooth(\GroupRep)$, the set of all vectors
    $\phi \in \PreHilb^{\textup{cpl}}$ for which the map
    $\group{SU}(1, n) \ni g \mapsto \skal{\psi}{\GroupRep(g)
      \phi}_{\PreHilb^{\textup{cpl}}} \in \CC$
    is smooth for all $\psi\in\PreHilb^{\textup{cpl}}$, and is defined as
    \begin{equation}
        \GroupRepDeriv(u) \phi
        :=
        \frac{\D}{\D t}\At{t=0}
        \GroupRep\big(\exp(tu)\big) \phi
    \end{equation}
    for all $\phi \in \Smooth(\GroupRep)$.
\end{theorem}
\begin{proof}
    By \cite[Thm.~5]{nelson:1959a} (see also
    \cite[Thm.~10.5.6]{schmuedgen:1990a}) we only have to show that the
    Nelson Laplacian $\Delta := \sum \frac{1}{\hbar^2}\MMdown(u_i)^2$, with
    $u_i$ running over a basis of $\lie{su}(1, n)$, is represented by an
    essentially self-adjoint operator. As the image of $\MMup$ is in
    $\PolynomeUnten[(1)]$ and Hermitian, it follows that
    $\Delta \in \PolynomeUnten[(2)]$, and as $\Delta$ is clearly
    bounded from below, we can apply Theorem~\ref{theorem:selfadjoint}.
\end{proof}

%
% An Additional $\ZZ_2$-Symmetry for $n = 1$
%

\subsection{\texorpdfstring{An Additional $\ZZ_2$-Symmetry for $n = 1$}{An
Additional Z2-Symmetry for n = 1}}
\label{subsec:AdditionalSymmetry}
The case $n=1$ seems to be of special importance due to some
interesting discrete symmetries. Here we can identify $\CC\PP^1$ with
$\cc{\CC} = \CC\cup\{\infty\}$, the one point compactification of
$\CC$, and thus $\Discext[1] \subseteq \CC\PP^1$ (via
$\embedInCPnExt$) with $\set{v\in \CC}{\abs{v}\neq1} \cup \{\infty\}$,
which makes the relation between $\Discstd[1] \cong \DD$ and
$\Discext[1]\cong \DD \mathbin{\dot{\cup}} \DD$ more apparent.
Note also that analogously,
$\DiscDouble[1] \subseteq \CC\PP^1\times\CC\PP^1$ (via
$\embedInCPnCpn$) is identified with $(\cc{\CC}\times\cc{\CC})\backslash
\big(\set{(u,v)\in \CC^2}{uv=1} \cup \{(0,\infty),(\infty,0)\} \big)$.

From the identification of $\Discext[1]$ with a subset of $\cc{\CC}$
one can already guess that there is another $\ZZ_2$-symmetry that we
did not discuss yet, which corresponds to the involution $\cc{\CC}\ni
w\mapsto-1/w\in\cc{\CC}$.
\begin{definition}
    Define the holomorphic involution $\invSpecial$ of
    $\DiscDouble[1]$ as
    \begin{equation}
        \invSpecial\colon
        [p^0,p^1,q^0,q^1]
        \; \mapsto \;
        \invSpecial([p^0,p^1,q^0,q^1])
        :=
        [-\I p^1, \I p^0, -\I q^1, \I q^0].
    \end{equation}
\end{definition}
It is not hard to check that this is indeed a well-defined holomorphic
involution of $\DiscDouble[1]$.  Moreover, it acts on functions in
$\Holomorphic(\DiscDouble[1])$ via pullback as $\invSpecial^*(\hat{a})
:= \hat{a}\circ \invSpecial$ and especially
\begin{equation}
    \label{eq:invSpecialAction}
    \invSpecial^*
    \big(\monomialDownHat{(P_0,P_1)}{(Q_0,Q_1)}\big)
    =
    (-\I)^{P_0+Q_0}\I^{P_1+Q_1}
    \monomialDownHat{(P_1,P_0)}{(Q_1,Q_0)}
    =
    (-1)^{P_1+Q_0} \monomialDownHat{(P_1,P_0)}{(Q_1,Q_0)}
\end{equation}
for all $P,Q\in\NN_0^2$ with $\abs{P}=\abs{Q}$. As
$\Holomorphic(\DiscDouble[1])$ is isomorphic to $\AnalyticUnten[1]$,
this of course yields an involution of $\AnalyticUnten[1]$ as
well. With respect to the standard chart this reads for the function
$\monomialDownHat{(P_0,P_1)}{(Q_0,Q_1)} =
\frac{w^{P_1}\cc{w}^{Q_1}}{(1-\abs{w}^2)^{\abs{P}}}$ as
\begin{equation}
    \label{eq:SpecialInvolutionOnF}
    \invSpecial^*\big(\monomialDown{(P_0,P_1)}{(Q_0,Q_1)}\big)
    =
    (-1)^{P_1+Q_0}\monomialDown{(P_1,P_0)}{(Q_1,Q_0)}
    =
    \frac{(-1)^{P_1+Q_0}w^{P_0}\cc{w}^{Q_0}}{(1-\abs{w}^2)^{\abs{P}}}
    =
    \frac{(-1/w)^{P_1}(-1/\cc{w})^{Q_1}}{(1-1/\abs{w}^2)^{\abs{P}}}
\end{equation}
i.e. this involution, transfered to $\AnalyticUnten[1]$, indeed
describes the above mentioned involution $w \mapsto -1/w$. It is also
worth noting that $\invSpecial$ commutes with $\involutionAll$, so its
pullback commutes with the $^*$-involution on
$\AnalyticUnten[1]$. With respect to the star product, we have:
\begin{proposition}
    \label{proposition:invequivariant}%
    Every $\starRed[\hbar]$ on $\AnalyticUnten[1]$ is
    $\invSpecial$-equivariant for all $\hbar\in\HbarDef$,
    i.e. $\invSpecial^*(a\starRed[\hbar]b) =
    \invSpecial^*(a)\starRed[\hbar]\invSpecial^*(b)$ holds for all
    $a,b\in\AnalyticUnten[1]$.
\end{proposition}
\begin{proof}
    As $\invSpecial$ is a holomorphic involution of $\DiscDouble[1]$,
    its pullback acts by a continuous linear function on
    $\AnalyticUnten[1]$ and so it is sufficient to check that
    $\invSpecial^*\big(\monomialDown{P}{Q} \starRed[\hbar]
    \monomialDown{R}{S}\big) = \invSpecial^*(\monomialDown{P}{Q})
    \starRed[\hbar] \invSpecial^*(\monomialDown{R}{S})$ holds for all
    $P, Q, R, S \in \NN_0^2$ with $\abs{P} = \abs{Q}$ and $\abs{R} =
    \abs{S}$, which is easily done using equations
    \eqref{eq:invSpecialAction} and \eqref{eq:starred}.
\end{proof}

%
% The bibliographies, make it smaller than other text
%

%
% this is the end of dstar
%


{
  \footnotesize
\begin{thebibliography}{10}

\bibitem {bayen.et.al:1978a}
\textsc{Bayen, F., Flato, M., Fr{{\o}}nsdal, C., Lichnerowicz, A., Sternheimer,
  D.: }\newblock \emph{Deformation Theory and Quantization}.
\newblock Ann. Phys.  \textbf{111} (1978), 61--151.

\bibitem {beiser:2011a}
\textsc{Beiser, S.: }\newblock \emph{A Convergent Algebra for the {W}ick Star
  Product on the {P}oincar{\'e} {D}isc}.
\newblock PhD thesis, Fakult{\"{a}}t f{\"{u}}r Mathematik und Physik,
  Physikalisches Institut, Albert-Ludwigs-Universit{\"{a}}t, Freiburg, 2011.
\newblock Available at {\url{https://www.freidok.uni-freiburg.de/data/8369}}.

\bibitem {beiser.roemer.waldmann:2007a}
\textsc{Beiser, S., R{\"o}mer, H., Waldmann, S.: }\newblock \emph{Convergence
  of the {W}ick Star Product}.
\newblock Commun. Math. Phys.  \textbf{272} (2007), 25--52.

\bibitem {beiser.waldmann:2014a}
\textsc{Beiser, S., Waldmann, S.: }\newblock \emph{Fr{\'{e}}chet algebraic
  deformation quantization of the Poincar{\'{e}} disk}.
\newblock Crelle's J. reine angew. Math.  \textbf{688} (2014), 147--207.

\bibitem {bertelson.cahen.gutt:1997a}
\textsc{Bertelson, M., Cahen, M., Gutt, S.: }\newblock \emph{Equivalence of
  Star Products}.
\newblock Class. Quant. Grav.  \textbf{14} (1997), A93--A107.

\bibitem {bieliavsky:2017a}
\textsc{Bieliavsky, P.: }\newblock \emph{Quantum differential surfaces of
  higher genera}.
\newblock Preprint  \textbf{arXiv:1712.06367} (2017), 46.

\bibitem {bieliavsky.gayral:2015a}
\textsc{Bieliavsky, P., Gayral, V.: }\newblock \emph{Deformation Quantization
  for Actions of {K}{\"a}hlerian Lie Groups}, vol. 236.1115 in \emph{Memoirs of
  the American Mathematical Society}.
\newblock American Mathematical Society, Providence, RI, 2015.

\bibitem {bordemann.brischle.emmrich.waldmann:1996a}
\textsc{Bordemann, M., Brischle, M., Emmrich, C., Waldmann, S.: }\newblock
  \emph{Phase Space Reduction for Star Products: An Explicit Construction for
  {$\mathbb C P^n$}}.
\newblock Lett. Math. Phys.  \textbf{36} (1996), 357--371.

\bibitem {bordemann.brischle.emmrich.waldmann:1996b}
\textsc{Bordemann, M., Brischle, M., Emmrich, C., Waldmann, S.: }\newblock
  \emph{Subalgebras with converging star products in deformation quantization:
  An algebraic construction for {$\mathbb C P^n$}}.
\newblock J. Math. Phys.  \textbf{37} (1996), 6311--6323.

\bibitem {cahen.gutt.rawnsley:1990a}
\textsc{Cahen, M., Gutt, S., Rawnsley, J.: }\newblock \emph{Quantization of
  K{\"{a}}hler Manifolds I: Geometric Interpretation of Berezin's
  Quantization}.
\newblock J. Geom. Phys.  \textbf{7} (1990), 45--62.

\bibitem {cahen.gutt.rawnsley:1993a}
\textsc{Cahen, M., Gutt, S., Rawnsley, J.: }\newblock \emph{Quantization of
  K{\"{a}}hler Manifolds. II}.
\newblock Trans. Am. Math. Soc.  \textbf{337}.1 (1993), 73--98.

\bibitem {cahen.gutt.rawnsley:1994a}
\textsc{Cahen, M., Gutt, S., Rawnsley, J.: }\newblock \emph{Quantization of
  K{\"{a}}hler Manifolds. III}.
\newblock Lett. Math. Phys.  \textbf{30} (1994), 291--305.

\bibitem {cahen.gutt.rawnsley:1995a}
\textsc{Cahen, M., Gutt, S., Rawnsley, J.: }\newblock \emph{Quantization of
  K{\"{a}}hler Manifolds. IV}.
\newblock Lett. Math. Phys.  \textbf{34} (1995), 159--168.

\bibitem {dewilde.lecomte:1983b}
\textsc{DeWilde, M., Lecomte, P. B.~A.: }\newblock \emph{Existence of
  Star-Products and of Formal Deformations of the Poisson Lie Algebra of
  Arbitrary Symplectic Manifolds}.
\newblock Lett. Math. Phys.  \textbf{7} (1983), 487--496.

\bibitem {esposito.stapor.waldmann:2017a}
\textsc{Esposito, C., Stapor, P., Waldmann, S.: }\newblock \emph{Convergence of
  the Gutt Star Product}.
\newblock J. Lie Theory  \textbf{27} (2017), 579--622.

\bibitem {fedosov:1994a}
\textsc{Fedosov, B.~V.: }\newblock \emph{A Simple Geometrical Construction of
  Deformation Quantization}.
\newblock J. Diff. Geom.  \textbf{40} (1994), 213--238.

\bibitem {gerstenhaber:1964a}
\textsc{Gerstenhaber, M.: }\newblock \emph{On the Deformation of Rings and
  Algebras}.
\newblock Ann. Math.  \textbf{79} (1964), 59--103.

\bibitem {gutt:1983a}
\textsc{Gutt, S.: }\newblock \emph{An Explicit $*$-Product on the Cotangent
  Bundle of a Lie Group}.
\newblock Lett. Math. Phys.  \textbf{7} (1983), 249--258.

\bibitem {kontsevich:2003a}
\textsc{Kontsevich, M.: }\newblock \emph{Deformation Quantization of {P}oisson
  manifolds}.
\newblock Lett. Math. Phys.  \textbf{66} (2003), 157--216.

\bibitem {masson.mcclary:1972a}
\textsc{Masson, D., McClary, W.~K.: }\newblock \emph{Classes of {$C^{\infty }$}
  vectors and essential self-adjointness}.
\newblock J. Funct. Anal.  \textbf{10} (1972), 19--32.

\bibitem {natsume:2000a}
\textsc{Natsume, T.: }\newblock \emph{{$C^*$}-algebraic deformation
  quantization of closed {R}iemann surfaces}.
\newblock In: \textsc{Cuntz, J., Echterhoff, S. (eds.): }\newblock
  \emph{{$C^*$}-algebras},   142--150,. Springer-Verlag, Berlin, 2000.
\newblock Proceedings of the SFB Workshop held at the University of
  M{\"u}nster, M{\"u}nster, March 8--12, 1999.

\bibitem {natsume.nest:1999a}
\textsc{Natsume, T., Nest, R.: }\newblock \emph{Topological Approach to Quantum
  Surfaces}.
\newblock Commun. Math. Phys.  \textbf{202} (1999), 65--87.

\bibitem {natsume.nest.peter:2003a}
\textsc{Natsume, T., Nest, R., Peter, I.: }\newblock \emph{Strict Quantizations
  of Symplectic Manifolds}.
\newblock Lett. Math. Phys.  \textbf{66} (2003), 73--89.

\bibitem {nelson:1959a}
\textsc{Nelson, E.: }\newblock \emph{Analytic vectors}.
\newblock Ann. of Math. (2)  \textbf{70} (1959), 572--615.

\bibitem {nest.tsygan:1995a}
\textsc{Nest, R., Tsygan, B.: }\newblock \emph{Algebraic Index Theorem}.
\newblock Commun. Math. Phys.  \textbf{172} (1995), 223--262.

\bibitem {omori.maeda.miyazaki.yoshioka:2000a}
\textsc{Omori, H., Maeda, Y., Miyazaki, N., Yoshioka, A.: }\newblock
  \emph{Deformation quantization of {F}r{\'e}chet-{P}oisson algebras:
  convergence of the {M}oyal product}.
\newblock In: \textsc{Dito, G., Sternheimer, D. (eds.): }\newblock
  \emph{Conf{\'e}rence Mosh{\'e} Flato 1999. Quantization, Deformations, and
  Symmetries}, \emph{Mathematical Physics Studies} no. \textbf{22},   233--245.
  Kluwer Academic Publishers, Dordrecht, Boston, London, 2000.

\bibitem {rieffel:1993a}
\textsc{Rieffel, M.~A.: }\newblock \emph{Deformation quantization for actions
  of $\mathbbm{R}^d$}.
\newblock Mem. Amer. Math. Soc.  \textbf{106}.506 (1993), 93 pages.

\bibitem {schmuedgen:1990a}
\textsc{Schm{\"{u}}dgen, K.: }\newblock \emph{Unbounded Operator Algebras and
  Representation Theory}, vol.~37 in \emph{Operator Theory: Advances and
  Applications}.
\newblock Birkh{\"{a}}user Verlag, Basel, Boston, Berlin, 1990.

\bibitem {schoetz.waldmann:2018a}
\textsc{Sch{\"o}tz, M., Waldmann, S.: }\newblock \emph{Convergent star products
  for projective limits of Hilbert spaces}.
\newblock J. Funct. Anal.  \textbf{274}.5 (2018), 1381--1423.

\bibitem {waldmann:2014a}
\textsc{Waldmann, S.: }\newblock \emph{A nuclear Weyl algebra}.
\newblock J. Geom. Phys.  \textbf{81} (2014), 10--46.

\end{thebibliography}

}
\end{document}